\documentclass[a4paper,11pt]{article}

\usepackage{amsmath,amssymb,dsfont}

\usepackage{upgreek,xspace}
\newcommand{\salg}{\ensuremath{\upsigma\text{-algebra}}\xspace}

\usepackage{hyperref}
\newcommand{\doi}[1]{\href{http://dx.doi.org/\detokenize{#1}}{doi: \detokenize{#1}}}

\usepackage[numbers]{natbib}
\usepackage[left=1cm, right=3cm, top=1cm, bottom=2cm]{geometry}

\newcommand{\deriv}[1]{\frac{\mathrm{d}}{\mathrm{d}{#1}}}
\newcommand{\pderiv}[1]{\frac{\partial}{\partial{#1}}}

\newcommand{\sX}{\ensuremath{\mathcal{X}}\xspace}
\newcommand{\sY}{\ensuremath{\mathcal{Y}}\xspace}
\newcommand{\dd}{\mathrm{d}}

\newcommand{\NN}{\mathbb{N}}

\newcommand{\RR}{\ensuremath{\mathbb{R}}}
\newcommand{\RRP}{\ensuremath{\RR^+}}
\newcommand{\II}{\mathds{1}}
\providecommand{\abs}[1]{\left\lvert{#1}\right\rvert}
\providecommand{\norm}[1]{\left\lVert{#1}\right\rVert}
\DeclareMathOperator*{\esssup}{ess\,sup}

\newcommand{\Lpb}{\ensuremath{B}\xspace}
\newcommand{\normi}[1]{{\norm{#1}}_{\sY-\infty}}
\newcommand{\normoi}[1]{{\norm{#1}}_B}
\newcommand{\normoid}[1]{{\norm{#1}}_{B^d}}
\newcommand{\Wopb}{\ensuremath{D}\xspace}
\newcommand{\normooi}[1]{{\norm{#1}}_D}
\newcommand{\Mi}{\ensuremath{\mathcal{M}}\xspace}
\newcommand{\Moi}{\ensuremath{\mathcal{M}_{0,\infty}}\xspace}
\newcommand{\Mbv}{\ensuremath{\mathcal{M}_\mathrm{BV}}}

\newcommand{\normmoi}[1]{\ensuremath{\norm{#1}_{\mathcal{M}_{0,\infty}}}\xspace}

\newcommand{\Cho}{\ensuremath{C_1}}
\newcommand{\Chd}{\ensuremath{C_2}}
\newcommand{\CAW}{\ensuremath{C_3}}
\newcommand{\CGro}{\ensuremath{C_4}}

 \usepackage{amsthm}
 \newtheorem{prop}{Proposition}
 \newtheorem{defn}[prop]{Definition}
 \newtheorem{thrm}[prop]{Theorem}

\makeatletter
\newcommand{\opnorm}{\@ifstar\@opnorms\@opnorm}
\newcommand{\@opnorms}[1]{%
  \left|\mkern-1.5mu\left|\mkern-1.5mu\left|
   #1
  \right|\mkern-1.5mu\right|\mkern-1.5mu\right|
}
\newcommand{\@opnorm}[2][]{%
  \mathopen{#1|\mkern-1.5mu#1|\mkern-1.5mu#1|}
  #2
  \mathclose{#1|\mkern-1.5mu#1|\mkern-1.5mu#1|}
}
\makeatother

\begin{document}

\title{Properties of the solutions of delocalised coagulation and inception problems with outflow boundaries}

\author{Robert I. A. Patterson\thanks{Weierstrass Institute, Mohrenstr. 39, 10117 Berlin, Germany. Robert.Patterson@wias-berlin.de}}
\date{February 19, 2015}

\maketitle
\begin{abstract}
Well posedness is established for a family of equations modelling particle populations
undergoing delocalised coagulation, advection, inflow and outflow in a externally specified velocity
field.  Very general particle types are allowed while the spatial domain is a bounded region of
$d$-dimensional space for which every point lies on exactly one streamline associated
with the velocity field.  The problem is formulated as a semi-linear ODE in the Banach space of 
bounded measures on particle position and type space.  A local Lipschitz property is established
in total variation norm for the propagators (generalised semi-groups) associated with the problem
and used to construct a Picard iteration that establishes local existence and global
uniqueness for any initial condition.  The unique weak solution is shown further to be a differentiable
or at least bounded variation
strong solution under smoothness assumptions on the parameters of the coagulation interaction.
In the case of one spatial dimension strong differentiability is established even for coagulation parameters
with a particular bounded variation structure in space.  This one dimensional extension establishes the
convergence of the simulation processes studied in [Patterson, \textit{Stoch. Anal. Appl.} 31, 2013]
to a unique and differentiable limit.


AMS 2010: 34G20, 35A01, 35A02, 35F61, 82C22
\end{abstract}

\section{Introduction}
Smoluchowski \citep{Smol16b} introduced equations for the concentrations of
particles of different sizes undergoing coagulation in a spatially homogeneous population.
\begin{equation}\label{e:smolold}
\deriv{t}c(t,y) =
\frac12 \sum_{y^\prime < y} K(y,y^\prime) c(t,y^\prime)c(t,y-y^\prime)
-c(t,y)\sum_{y^\prime} K(y,y^\prime) c(t,y^\prime),
\end{equation}
where $c(t,y)$ is the concentration of particles of size $y$ at time $t$ and $K$ is a symmetric
function defining the `reaction' rates.
The Smoluchowski coagulation equations can be regarded as describing a system of binary
reactions involving an infinite number of species, but with a very structured, although
non-sparse set of rates and \eqref{e:smolold} abstractly written $\dot{c}=R(c)$.  The model
therefore extends naturally to a reaction--transport problem for spatially inhomogeneous
populations of coagulating particles of the general form $\dot{c} + \mathcal{A}c=R(c)$
for some transport operator $\mathcal{A}$.

Since coagulation is a binary reaction in which every possible pair of particles may coagulate,
the equations are, even in the spatially homogeneous case, non-linear and more
significantly non-local in particle size (size may here be generalised to `type').  The
first existence results for the Smoluchowski coagulation equation and its extensions
were based on convergent sub-sequences of approximating stochastic processes.
The first convergence result of this kind with simple
diffusive transport of particles is due to \citet{Lang80}, generalisations were achieved
by \citet{Nor04,Nor14}, \citet{Wells06b} and \citet{Yag09}.  This is quite a natural approach, because
the equations are based on a microscopic stochastic model and related stochastic processes
have also proved fruitful for numerical purposes going back to \citet{Marc68} and \citet{Gil72}.

The results just mentioned are essentially compactness results and say nothing about
uniqueness of the limiting trajectories, much less of uniqueness for the solutions to the
Smoluchowski equation and its extensions.  Convergence and uniqueness were proved
together by \citet{Gui01} who modelled diffusion as a random walk on a lattice and used a
more functional analytic approach.  Going further in this direction one is led to regard the
Smoluchowski equation and its extensions as an ODE on a Banach space and to proceed
via a locally Lipschitz source term and a Picard iteration method to show existence and
uniqueness in some functional setting.  The general strategy is presented in chapters 5\&6 of
\citep{Pazy83}.  Applications to Smoluchowski problems are given by \citep{Wro02,Ama05,Bai11}
and the works cited therein.

Especially when approaching the Smoluchowski equation from the point of view of stochastic
particle systems it is natural to think of measure valued solutions.  A particle system is identified
with its empirical measure and thus instead of functional solutions one is led to look at measure
valued solutions in a weak setting.  To give the concrete example that will be the focus of this work:
A solution (with a given initial condition) is a flow of measures
$\mu_t$ on positions in \sX and particle types (sizes and potentially additional details) in \sY
satisfying
\begin{multline}\label{eq:smol}
\deriv{t}\int_{\sX\times\sY}f(x,y)\mu_t(\dd x,\dd y) \\
= \int_{\sX\times\sY}u_t(x)\cdot\nabla f(x,y)\mu_t(\dd x,\dd y)+
 \int_{\sX\times\sY}f(x,y)I_t(\dd x,\dd y) \\
+ \frac12\int_{(\sX\times\sY)^2}\left[f(x_1,y_1+y_2)-f(x_1,y_1)-f(x_2,y_2)\right]\\
   K(y_1,y_2)h(x_1,x_2)\mu_t(\dd x_1,\dd y_1)\mu_t(\dd x_2,\dd y_2)
\end{multline}
for all $f$ in a class of functions $D$ to be specified below.  Here the problem has been moved
from the strong formulation of \eqref{e:smolold} to a weak setting; a transport operator $u_t \cdot \nabla$
(the dual of the $\mathcal{A}$ mentioned above) has been introduced and the delocalisation
of the coagulation specified via a function $h$, which may be regarded as a mollifier.
A particle source term $I_t$ has also been added, which is relevant for many real-world applications
as discussed later.

Signed measures can be regarded as Banach space under a wide range of norms and equation~\eqref{eq:smol}
interpreted as a Banach space valued ODE and Picard-like fixed point strategies introduced.
An important insight of the monograph \citep{Kolok10b} was to exploit duality of linear
operators and norms between measures and appropriate spaces of test functions in pursuit of 
this programme.  In this way
one performs most calculations for operators on test function spaces, which are a little easier to
work with than operators on spaces of measures.  Measure valued solutions are also the topic
of \citep{Nor14}, which also uses a linear operator approach, but uses approximation rather than
duality arguments and deals with unbounded coagulation kernels.

All the work discussed so far deals with diffusing particles (contrast \eqref{eq:smol})
and solutions either with a zero gradient
boundary conditions, which excludes outflow or defined on the whole of $\RR^d$ so that outflow
is thereby excluded.  For numerical reasons motivated by applications in engineering, the present
author has been interested in the Smoluchowski equation with advective transport and a delocalised
coagulation interaction \citep{Pat12,Menz12}.  In particular for engineering applications particle gain
and loss terms are important---industrial equipment is designed to take in material, alter it and then
send it on either as waste or product.  This gives the problem as formulated in \citep{Pat12,Menz12}
and other applied works a different structure to those studied in previous mathematical
works.  For example, individual particles experience irreversible processes, but nevertheless the
system is expected to reach a steady state in the large time limit under a wide range of conditions.
Measure valued processes (which can be interpreted as particle processes) with an inflow term although
no interaction were also studied in \citep{Eve12}.

For \eqref{eq:smol} specific problem an
initial existence result via the compactness of approximating stochastic processes was given in
\citep{Pat13}.  In that work however convergence of the approximating processes could not be proved,
only sequential compactness, because the number of distinct limit points was unknown.  This was
not only mathematically frustrating, but also a major obstacle hindering the numerical analysis of the
associated simulation methods.

The purpose of the present work is to establish uniqueness of measure valued solutions for \eqref{eq:smol}.
Additionally Lipschitz continuity in the initial conditions is shown and the same Picard iteration
method that proves uniqueness of solutions provides a purely analytic existence proof.  The result
can thus be characterised as one of ``well posedness''.  Formally there are some new existence results---the assumption of only one spatial
dimension in \citep{Pat13} is relaxed, but with the assumptions used in this work the proof in that paper
could easily be extended.  The existence of a differentiable strong solution is of interest,
because it opens the way to a study of the way in which the solution approaches a solution of the
corresponding equation with a local coagulation interaction, see for example \citep{Pat14}.

\section{Statement of Main Results}

In order to make a precise statement it is first necessary to go into details regarding the various
objects appearing in \eqref{eq:smol}.  The basic spaces are the particle position and type spaces
\sX and \sY respectively.  The type space, which carries information about the mass and any
other internal details of a particle is assumed to be a locally compact, second countable Hausdorff
space on which
coagulation is represented by a commutative +  operator.  The particle position space $\sX$ is
assumed to be a simply connected, relatively compact subset of $\RR^d$, which is equipped
with Lebesgue measure and a derivative $\nabla$.
Both \sX and \sY are given their respective Borel {\salg}s and $\sX \times \sY$
is given the product topology and \salg.

Throughout this work $\RR^d$ will be given the usual Euclidean norm, which will be written
$\abs{\cdot}$.  Linear operators $L$ between two normed spaces $\left(A,\norm{\cdot}_A\right)$
$\left(B,\norm{\cdot}_B\right)$ are given the operator norm
\begin{equation*}
\norm{L}_{A\rightarrow B} := \sup_{x\in A \colon \norm{x}_A=1}\norm{Lx}_B.
\end{equation*}

\subsection{Properties of the Flow and Spatial Domain}\label{s:flowass}
\subsubsection*{Velocity field}
Particles are assume to be transported in a time dependent velocity field $u_t$ defined on
$\overline{\sX}$ the closure
of \sX such that $u \in C\left(\RRP, C^2\left(\overline{\sX},\RR^d\right)\right)$, satisfying
\begin{itemize}
\item $\norm{u}_\infty :=\sup_{t\in\RRP,x\in\overline{\sX}}\abs{u_t(x)} < \infty$,
\item $\norm{\nabla \cdot u}_\infty :=\sup_{t\in\RRP,x\in\overline{\sX}}
            \abs{\sum_{k=1}^d\pderiv{x_k}u_{k,t}(x)} < \infty$,
\item $\opnorm{\nabla u} := \sup_t \norm{\nabla u_t(x)}_{\RR^d \rightarrow \RR^d}  < \infty$ viewing the
matrices $\nabla u_t$ as linear operators,
\item $\norm{\nabla \nabla \cdot u}_\infty :=\sup_{t\in\RRP,x\in\overline{\sX}}\abs{\nabla \left(\nabla\cdot u_{t}(x)\right)} < \infty$.
\end{itemize}

\subsubsection*{Boundaries}
It is assumed that the spatial domain \sX is simply connected and has a regular boundary $\partial \sX$
that can be decomposed into three parts, each with outward normal $n(x)$:
\begin{itemize}
\item $\Gamma_\mathrm{in}$ where $n(x)\cdot u_t (x)<0$ for all $t\in\RRP$,
\item $\Gamma_\mathrm{side}$ where $n(x)\cdot u_t(x) = 0$,
\item $\Gamma_\mathrm{out}$ where
$n(x)\cdot u_t (x)>0$.  
\end{itemize}
Further $\Gamma_\mathrm{in} \subset \sX$ but
$\Gamma_\mathrm{side},\Gamma_\mathrm{out}  \subset \RR^d \setminus \sX$.

\subsubsection*{Flow Field}
Define $\Phi_{s,t}(x)$ as the position at time $t$ of a particle moving with the velocity field $u$ starting
from $x$ at time $s$.  It is assumed that
\begin{itemize}
\item There exists a $t_0>0$ such that, for all $t\geq 0$ and $x\in\sX$ one has $\Phi_{t,t+t_0}(x) \notin \sX$,
          that is, an upper bounded on the residence time.
\item For every $t>0$ and $x\in\sX$ there exist unique $s(t,x)$, $\xi(t,x)$ such that
         $\Phi_{s(t,x),t}\left(\xi(t,x)\right)=x$ and either $s(t,x)=0$ or $\xi(t,x)\in\Gamma_\mathrm{in}$
         (the possibility of both is not excluded).  This defines a start position for each point in the flow
         and $\xi(t,x) = \Phi_{t,s(x)}(x)$.
\item $s(t,x)$ and $\xi(t,x)$ are differentiable in $x$ and
          $\norm{\nabla s}_\infty :=\sup_{t,x} \abs{\nabla s(t,x)} < \infty$.  A bound for the derivative of
          $\xi$ is given in the appendix.
\item The set $\Xi_t = \left\{x\in\sX \colon \xi(t,x)\in\Gamma_\mathrm{in}\right\}$ forms a differentiable
          $d-1$ dimensional manifold that divides $\sX\setminus\Xi_t$ into two disjoint simply connected
          components.
\end{itemize}

\subsection{Test Function Spaces}
\begin{defn}\label{d:BY}
Let $\mathcal{B}_\mathrm{b}(\sY)$ be the space of bounded measurable functions on $\sY$
with the supremum (not essential supremum) norm, which will be written 
$\normi{\cdot}$.
\end{defn}
\begin{defn}\label{d:B}
Let $B:=\mathcal{B}_\mathrm{b}(\sX\times\sY)$ be the space of bounded measurable functions on
 $\sX\times\sY$ with the
supremum (not essential supremum) norm, which will be written 
$\normoi{\cdot}$.
\end{defn}
\begin{defn}\label{d:Bd}
Let $B^d$ be the space of $d$-dimensional vector valued functions with components in \Lpb.
This will be given the norm
$\normoid{f} := \sup_{x,y}\abs{f(x,y)}$, where $\abs{\cdot}$ is the Euclidean norm on $\RR^d$.
\end{defn}

To handle the derivative in \eqref{eq:smol} and associated boundary condition introduce
\begin{defn}\label{d:D}
\begin{equation*}
\Wopb :=
\left\{f \in B \colon f \text{ differentiable }, \nabla f \in B^d, \lim_{x \rightarrow \Gamma_\mathrm{out}}  \norm{f(x,\cdot)}_{\sY-\infty} = 0\right\}
\end{equation*}
The norm is
\begin{equation*}
\normooi{f} = \normoi{f} + \normoid{\nabla f}.
\end{equation*}
\end{defn}
This is an appropriate class of test functions to use in \eqref{eq:smol}, because the derivative
is well behaved.
For a discussion of the boundary condition see \citep{Pat13}, although that work imposes
slightly stricter regularity conditions, which are here seen to be unnecessary.

\subsection{Solution Spaces}

A particle distribution is at a minimum a measure on the product of the particle position and type spaces,
that is on $\sX\times\sY$.  The solution processes must accordingly take values in the following spaces,
which are built from the space of measures on particle types \sY:

\begin{defn}\label{d:normtv}
Let $\left(\mathcal{M}(\sY),\norm{\cdot}_{\sY-\mathrm{TV}}\right)=:\mathcal{M}(\sY)_{\mathrm{TV}}$
be the normed space of signed bounded measures on $\sY$ with the total variation norm
\begin{equation*}
\norm{\mu}_{\sY-\mathrm{TV}}:=\sup_{f\neq 0} \frac{\abs{\int_\sY f(y)\mu(\dd y)}}{\norm{f}_{\sY-\infty}},
\quad
f \in \mathcal{B}_\mathrm{b}(\sY).
\end{equation*}
\end{defn}

\begin{defn}\label{d:Mi}
Let $\Mi = \mathcal{M}(\sX\times \sY)$ be the vector space of bounded signed measures on $\sX\times \sY$.
\end{defn}
Under reasonable assumptions one expects to find solutions to \eqref{eq:smol} that are absolutely
continuous with respect to Lebesgue measure on \sX; this leads to the following space (compare  \citep{Nor14}).
\begin{defn}\label{d:Moi}
\begin{equation*}
\Moi = \left\{\mu\in\Mi \colon \mu(\dd x, \dd y) = c(x,\dd y)\dd x,
 c \in L^\infty\left(\sX,\mathcal{M}(\mathcal{\sY})_\mathrm{TV}\right) \right\}
\end{equation*}
with the norm
\begin{equation*}
\normmoi{c} =  \esssup_x \norm{c(x,\cdot)}_{\sY-\mathrm{TV}}
\end{equation*}
where a measure is identified with its density.
\end{defn}

The $B$ and $D$ dual norms on \Mi will play a role in this work 
\begin{defn}\label{d:normds}
Let $\mu \in \Mi$,
\begin{equation*}
\norm{\mu}_\mathrm{TV} \equiv
\norm{\mu}_{B^\star} := \sup_{f\neq 0} \frac{\abs{\int_{\sX\times\sY} f(x,y)\mu(\dd x,\dd y)}}{\normoi{f}},
\quad
f \in \Lpb,
\end{equation*}
and
\begin{equation*}
\norm{\mu}_{\Wopb^\star} := \sup_{f\neq 0} \frac{\abs{\int_{\sX\times\sY} f(x,y)\mu(\dd x,\dd y)}}{\normooi{f}},
\quad
f \in \Wopb.
\end{equation*}
\end{defn}
As the notation suggests, the $B^\star$ norm is the total variation norm on $\Mi$.  For calculations the
$B^\star$ point of view is emphasised, however the main results are stated in terms of TV. 
When dealing with processes the following abbreviation is useful
\begin{defn}\label{d:nnormbs}
Let $T>0$ and  $c\in L^\infty\left([0,T);\Mi \right)$ then
\begin{equation*}
\opnorm{c}_{B^\star} := \esssup_{t\in[0,T)}\norm{c(t,\cdot,\cdot)}_{B^\star}.
\end{equation*}
\end{defn}

\subsection{Coagulation}\label{s:coag}
It is now possible to set out the assumptions on the coagulation dynamics specified by  $K$ and $h$
in \eqref{eq:smol}.  $K$ is assumed to be non-negative and measurable with some bound $K_\infty>0$
such that $\sup_{y_1,y_2}K(y_1,y_2) \leq K_\infty$.

The delocalisation $h\colon \sX^2 \rightarrow \RR$ must be measurable and non-negative.
For fixed $x_1 \in \sX$ write $h_{1,x_1}$ and $h_{2,x_1}$ for the the functions given by
$h_{1,x_1}(\cdot)=h(x_1,\cdot)$ and $h_{2,x_1}(\cdot)=h(\cdot,x_1)$.  It will be assumed that neither
$K$ nor $h$ are identically zero---this would lead to a trivial problem with no coagulation.

H1: $\normoi{h_{i,x}} \leq \Cho\quad \forall x \in \sX,\ i=1,2$.

H2: H1 holds and $h(x,x_2)=\sum_{j=1}^J \chi_{j,1}(x)\chi_{j,2}(x_2)$ with $\chi_{j,i}$ positive,
       and of special bounded variation (derivative in $L^1$ plus atoms) for all $i$ and $j$
       with the number of atoms in the weak derivatives
       bounded.   Further one has
       $\sup_{x,x_2} \sum_{j=1}^J \normoi{\chi_{j,2}}\int_r^t
                     \abs{\nabla \chi_{j,1}\left(\Phi_{r,s}(x),x_2\right)}\dd s
                     \leq \Chd t_0 e^{\opnorm{\nabla u} \min(t-r,t_0)}$
        and its symmetric counterpart
       $\sup_{x_1,x} \sum_{j=1}^J \normoi{\chi_{j,1}}\int_r^t
                     \abs{\nabla \chi_{j,2}\left(x_1,\Phi_{r,s}(x)\right)}\dd s
                     \leq \Chd t_0 e^{\opnorm{\nabla u} \min(t-r,t_0)}$
                     
H3: H1 holds and the $h_{i,x}$ are in \Wopb with$\normoi{\pderiv{\xi}h_{i,x}(\xi)} \leq \Chd\ \forall x \in \sX, i=1,2$.  It should be noted that H3 implies H2 (Proposition~\ref{p:dxPhi2} is helpful here).

The function $h$ parametrises the numerical methods that lie behind this work \citep{Pat12}.  H2 is
describes the case where the spatial domain is partitioned into cells and coagulation is only simulated
between particles that are in the same cell.  From a software point of view this is somewhat simpler than
dealing with functions satisfying H3.  In one dimension, which was the case simulated in \citep{Pat12}, 
H2 is a weak integrability condition on the derivative of $h$.

\subsection{Inception}\label{s:incep}

Particles are added to the system with intensity given by signed measures $I_t \in \Mi$.

I1: $\sup_t \norm{I_t}_{\Lpb^\star} < \infty$ and $I \in C([0,\infty),\left(\Mi,\norm{\cdot}_{D^\star}\right))$.

I2: I1 holds, the $I_t$ are non-negative measures and for every $f\in\Lpb$
\begin{equation*}
\int_{\sX\times\sY}f(x,y)I_t(\dd x, \dd y) =
\int_{\sX\times\sY}f(x,y)I_\mathrm{int}(t,x, \dd y)\dd x
+ \int_{\Gamma_\mathrm{in}\times\sY}f(\xi,y)I_\mathrm{bdry}(t,\xi, \dd y)\dd \xi
\end{equation*}
with $I_\mathrm{int}\in C\left([0,\infty),\Moi\right)$ also
$I_\mathrm{bdry} \in C\left([0,\infty),L^\infty\left(\Gamma_\mathrm{in},\mathcal{M}(\mathcal{\sY})_\mathrm{TV}\right)\right)$
with the respective norms uniformly bounded for all time and with some $I_\ast>0$ such that
$\norm{I_\mathrm{bdry}(t,\xi, \cdot)}_{\sY-\mathrm{TV}} \leq I_\ast u_t(\xi)\cdot n(\xi)$ for all $t$ and
$\xi\in\Gamma_\mathrm{in}$.

I3: I2 holds, $I_\mathrm{bdry}$ has a time and space derivative so that $I_\mathrm{bdry} \in C^1\left([0,\infty)\times \Gamma_\mathrm{in},\mathcal{M}(\mathcal{\sY})_\mathrm{TV}\right)$ and
$I_\mathrm{int}$ has an \sX-derivative which is $\nabla I_\mathrm{int}\in C\left([0,\infty),(\Moi)^d\right)$

\subsection{Statements of the Theorems}\label{s:results}
These results progress from local existence and uniqueness of a measure valued solution to a global result and then existence followed by differentiability of a density for the measures.

\begin{thrm}\label{t:locexist}
Assume H2 or H3 holds and that $c_0 \in \Mi$, then there exists a $T=T(c_0)$ such there is a
unique solution $c_t$ to \eqref{eq:smol} in $L^\infty\left([0,T),\left(\Mi,\norm{\cdot}_\mathrm{TV}\right)\right)$ with initial condition $c_0$ and this solution is in
$C\left([0,T),\left(\Mi,\norm{\cdot}_\mathrm{TV}\right)\right) \cap C^1\left((0,T),\left(\Mi,\norm{\cdot}_{D^\star}\right)\right)$.

Additionally, there is no time interval on which more than one TV-bounded solution exists for
a given initial condition.  If solutions exist on a common compact time interval for at two or more
initial conditions, then the solutions are Lipschitz continuous with respect to the initial
data in the TV-norm on this compact time interval.
\end{thrm}

In the physically reasonable setting of non-negative particle numbers, the previous result holds
for all time:
\begin{thrm}\label{t:globexist}
The $T=T(c_0)$ from the previous theorem is $\infty$ if $c_0$ and the $I_t$ are non-negative measures.
\end{thrm}

\begin{thrm}\label{t:globdens}
Assume H2 or H3 holds, that $c_0$ is in the positive cone of \Moi, and that I2 is satisfied,
then \eqref{eq:smol} has a unique solution,
which is in $L^\infty\left([0,\infty),\Moi\right)$ and therefore has a density in
$L^\infty\left([0,\infty)\times\sX,\mathcal{M}(\sY)_\mathrm{TV}\right)$ starting from $c_0$.
\end{thrm}

\begin{thrm}\label{t:globdiff}
Assume that $c_0 \in W^{1,\infty}\left(\sX,\mathcal{M}(\sY)_\mathrm{TV}\right)$ is consistent with the
boundary condition given below, that I3 is satisfied, and further
that either H3 holds and \sX has a sufficiently regular boundary  or $d=1$, H2 holds and $u$ is bounded
away from 0, then \eqref{eq:smol} has a unique solution $c$ with a density in
$W^{1,\infty}\left([0,\infty)\times\sX,\mathcal{M}(\sY)_\mathrm{TV}\right)$, satisfying the boundary condition
$-u_t(x)\cdot n(x) c(t,x,\dd y) = I_\mathrm{bdry}(t,x,\dd y)\quad \forall t\in\RRP,\ x\in\Gamma_\mathrm{in}$
and with initial condition $c_0$.
\end{thrm}

As a corollary of the preceding two results an earlier result by the author, which demonstrated the
existence of converging sub-sequences of stochastic approximations to solutions \eqref{eq:smol} can be
extended to a full convergence result:
\begin{thrm}\label{t:stochconv}
The stochastic jump processes studied in \citep{Pat13}, which have $\sX=[0,L)$ for some $L>0$
and satisfy H2 and I3 converge to the unique solution of \eqref{eq:smol} and this weak solution
is also a strong solution in the Sobolev
space $W^{1,\infty}\left([0,\infty)\times[0,L),\mathcal{M}(\sY)_\mathrm{TV}\right)$
provided that the initial condition $c_0$ is in $W^{1,\infty}\left([0,L),\mathcal{M}(\sY)_\mathrm{TV}\right)$
with $u_0(0)c_0(0,\dd y) = I_\mathrm{bdry}(0,0,\dd y)$.  Further one has
$u_t(0)c(t,0,\dd y) = I_\mathrm{bdry}(t,0,\dd y)$ for all $t$.
\end{thrm}
\begin{proof}
In \citep{Pat13} it was shown that every sequence of approximating processes has a sub-sequence
converging to a solution of \eqref{eq:smol} (a compactness result).  Theorem~\ref{t:globexist} shows
that there is only one such limit point so one has convergence and Theorem~\ref{t:globdiff} yields the
differentiability.
\end{proof}

\section{Dual Operator Estimates}\label{s:dual}
Introduce the more compact notation $\left<f,\mu\right>=\int_{\sX\times\sY}f(x,y)\mu(\dd x, \dd y)$
for $f \in B$ and $\mu\in \Mi$.  It is now helpful to seek a generator for the evolution given
in \eqref{eq:smol}, that is an operator $A_t$ such that
\begin{equation}\label{eq:Atstar}
\deriv{t}\left<f,\mu_t\right> = \left<A_t(f),\mu_t\right> +\left<f,I_t\right>.
\end{equation}
This is in fact a dual generator, because it acts on the functions not the measures.

The author emphasises his dependence on \citet{Kolok10b} for the material in this section and the first half
of the next.  The first novelty in this section is the boundary condition associated with the finite domain and
outflow, which required careful
treatment, but is not covered by the existing work.  Also the consideration of coefficients of bounded
variation (H2) is essential to treating the motivating example from \citep{Pat13} and even under this
relatively weak assumption differentiability of the solutions in one spatial dimension is established.
An additional variation from \citep{Kolok10b} appears in Proposition~\ref{p:lipsch} where some additional
problem structure is exploited and enables the fixed point methods to be applied in the $\Lpb^\star$-norm,
rather than the weaker $\Wopb^\star$-norm used in \citep{Kolok10b}.

\subsection{The Generators}
Because \eqref{eq:smol} is quadratic in $\mu$ the same must be true of the expression
$\left<A_t[\mu](f),\mu_t\right>$, which is achieved by including the path $(\mu_r)_{r\in[0,t]}$
as a parameter of $A_t$.  It is technically convenient to parametrise by the entire path, not
just $\mu_t$, because one eventually deals with propagators where the dependence cannot be
expressed in terms of $\mu$ at any finite set of time points.  One notes that
$A_t[\mu] = U_t + H_t[\mu]$ where $U$ is the transport operator and $H$ is the coagulation
operator.

\begin{defn}\label{d:Htstar}
Let $\mu\in L^\infty\left([0,T),\left(\Mi,\norm{\cdot}_{B^\star}\right)\right)$ and assume H1 holds.
The coagulation generator parametrised by $\mu$ is
$H_t[\mu] \colon B \rightarrow B$ defined by
\begin{multline}
H_t[\mu](f)(x,y)
=\frac12\int_{\sX\times\sY}h(x,x_2)f(x,y+y_2)K(y,y_2)\mu_t(\dd x_2,\dd y_2)\\
-\frac12\int_{\sX\times\sY}f(x,y)\left[h(x,x_2)+h(x_2,x) \right]K(y,y_2)\mu_t(\dd x_2,\dd y_2)
\end{multline}
for $t \in [0,T)$.
\end{defn}
This is not the only possible definition for $H_t[\mu]$, other versions also yield the desired
expression (the coagulation term from \eqref{eq:smol}) for $\left<H_t[\mu](f),\mu_t\right>$.
Each definition would lead to characterising the solutions as fixed points of a different mapping;
the definition given here seems to be the one that minimises the technical difficulties in the
following analysis.

\begin{prop}\label{p:Hbound}
Let $0 < T$ and $\mu\in L^\infty\left([0,T),\Mi\right)$ and assume H1 holds, then the operator
norm of $H_t[\mu]$ as a mapping $B \rightarrow B$ satisfies
\begin{equation*}
\esssup_t \norm{H_t[\mu]}_{B \rightarrow B} \leq \frac32 K_\infty  \Cho \opnorm{\mu}_{B^\star}.
\end{equation*}
\end{prop}
\begin{proof}
Immediate.
\end{proof}

\begin{defn}\label{d:tgen}
Let $t \in \RR$ and $f \in \Wopb$ then the transport generator $U_t \colon \Wopb \rightarrow \Lpb$
is given by
\begin{equation*}
U_tf(x,y) = u_t(x) \cdot \nabla f(x,y).
\end{equation*}
\end{defn}
One can now define $A_t[\mu] = U_t + H_t[\mu]$ as a linear operator $\Wopb \rightarrow \Lpb$.

\subsection{The Propagators}
Propagators are generalisations of semi-groups to deal with time dependent generators.
For a detailed discussion the reader is referred to \citep[Chapter 2]{Kolok10b} or
\citep[Chapter 5]{Pazy83}.  The key idea (given in the dual setting appropriate to this section)
is that a generator
$A_t$ generates a family of linear operators $A^{r,s}$ such that $A^{r,s}A^{s,t}=A^{r,t}$ and 
\begin{equation}
\deriv{t}A^{s,t}=A^{s,t}A_t, \quad \deriv{s}A^{s,t}=-A_sA^{s,t}.
\end{equation}
The goal of this section is to construct such a family of propagators for the generator $A_t[\mu]$
from the previous section.

\begin{defn}\label{d:tprop}
Let $s,t \in \RR$ and $f \in \Lpb$ and define the transport propagators
$U^{t,s}\colon \Lpb \rightarrow \Lpb$ by
\begin{equation*}
U^{s,t}f(x,y) =
\begin{cases}
f\left(\Phi_{s,t}(x),y\right) \quad &\Phi_{s,t}(x) \in \sX\\
0 &\text{otherwise}
\end{cases}
\end{equation*}
where $\Phi$ is the flow due to the velocity field $u$ (see \S\ref{s:flowass} and Appendix~\ref{s:flow}).
\end{defn}

\begin{prop}\label{p:Ubound}
Let $t \geq s$, then the transport propagator $U^{s,t}$ preserves \Wopb. The following operator norm estimates hold:
\begin{equation*}
\norm{U^{s,t}}_{\Lpb\rightarrow\Lpb} \leq \II\left(t-s\leq t_0\right),
\end{equation*}
where $\II$ is an indicator function and
\begin{equation*}
\norm{U^{s,t}}_{\Wopb\rightarrow\Wopb} \leq e^{(t-s) \opnorm{\nabla u}}\II\left(t-s\leq t_0\right).
\end{equation*}
\end{prop}
\begin{proof}
The \Lpb-norm of $f$ is immediate.  Use the chain rule and Proposition~\ref{p:dxPhi2} in the appendix for the derivative
of $f$.
\end{proof}

The required propagator is now constructed as a perturbation of the transport propagator $U$
by the bounded coagulation generator:
\begin{defn}\label{d:Asprop}
Let $T>0$, $0\leq r \leq t \leq T$ and $\mu\in L^\infty\left([0,T),\Mi\right)$, and define
(compare \citep[Theorem 2.9]{Kolok10b}):
\begin{multline*}
A^{r,t}[\mu] := 
U^{r,t}+
\sum_{m=1}^\infty \int_{r\leq s_1 \leq \dotsc \leq s_m\leq t}
 U^{r,s_1}H_{s_1}[\mu] \dotsm U^{s_{m-1},s_m} H_{s_m}[\mu] U^{s_m,t}
 \dd s_1 \dotsm \dd s_m.
\end{multline*}
\end{defn}

It is now necessary to establish estimates for the operator norm of $A$ on \Lpb and \Wopb. For this it is
shown that the infinite sum just given is absolutely convergent in both operator norms.  During this analysis
it is convenient to use some additional notation:
\begin{defn}\label{d:Aterm}
Under the assumptions of Definition~\ref{d:Asprop} let $f\in \Lpb$ and $t\geq 0$; define both
$f^0_{r,t} := U^{r,t}f$ and
\begin{equation*}
f_{r,t}^m := \int_r^{t} U^{r,s}H_s[\mu]f_{s,t}^{m-1}\dd s.
\end{equation*}
\end{defn}
This allows one to write
\begin{equation}\label{eq:Aspropsum}
A^{r,t}[\mu]f = \sum_{m=0}^\infty f_{r,t}^m.
\end{equation}

\begin{prop}\label{p:zero}	
Under the assumptions of Definitions~\ref{d:Asprop}\&\ref{d:Aterm}
\begin{equation*}
\norm{f_{r,t}^m(x,\cdot)}_{\sY-\infty}
\leq
\frac{1}{m!}\left(\frac32 K_\infty \Cho \opnorm{\mu}_{B^\star}(t-r)\right)^m \norm{f\left(\Phi_{r,t}(x),\cdot\right)}_{\sY-\infty},
\end{equation*}
which is zero for $t-r\geq t_0$ and
\begin{equation*}
\normoi{f_{r,t}^m}
\leq
\frac{1}{m!}\left(\frac32 K_\infty \Cho \opnorm{\mu}_{B^\star}(t-r)\right)^m \normoi{f}.
\end{equation*}
\end{prop}
\begin{proof}
Proceed by induction.
\end{proof}
The \Lpb-operator norm estimate now follows:
\begin{prop}\label{p:ApropB}
Let $T>0$, $0\leq r \leq t < T$ and $\mu\in L^\infty\left([0,T),\left(\Mi,\norm{\cdot}_{B^\star}\right)\right)$
and assume H1 holds, then $A^{r,t}[\mu]$ is a locally bounded propagator on
\Lpb satisfying 
\begin{equation*}
\norm{A^{r,t}[\mu]}_{B\rightarrow B}
   \leq e^{\frac32 K_\infty \Cho \opnorm{\mu}_{B^\star} (t-r)}\II(t-r \leq t_0)
\end{equation*}
and $A^{r,t}[\mu] f$ is $\normoi{}$-continuous in $t$ for every $f\in\Lpb$ and $t\geq r$ (this
is known as `strong continuity').
Further, for any $f \in D$ and almost all $t$
\begin{equation*}
\deriv{t}A^{r,t}[\mu]f=A^{r,t}[\mu]A_t[\mu]f,
\end{equation*}
where one recalls $A_t[\mu] = u_t \cdot \nabla + H_t[\mu]$.
\end{prop}
\begin{proof}
The first part of the result follows from Proposition~\ref{p:zero} and \eqref{eq:Aspropsum}.

The (left) generator $A_t[\mu]$ can be found differentiating the series in Proposition~\ref{d:Asprop} term
by term and observing that the resulting series is again absolutely convergent.
\end{proof}

Differentiating with respect to $r$ in Proposition~\ref{p:ApropB} is not possible, because $A^{r,t}$ does
not necessarily preserve \Wopb.  This is addressed in the next few propositions by making stronger
smoothness assumptions on $h$, the spatial delocalisation of the coagulation interaction introduced
in \S\ref{s:coag}  and used in the definition of $H$ (Definition~\ref{d:Htstar}).
\begin{prop}\label{p:Aderivterm}
Under the assumptions of Definitions~\ref{d:Asprop}\&\ref{d:Aterm} and additionally assuming
either H3 holds or H2 holds and $\mu$ is bounded for all (not just Lebesgue almost all $t$),
for example because it is continuous
\begin{multline*}
\norm{\nabla f_{r,t}^m(x,\cdot)}_{\sY-\infty}
\leq
\left(\frac32 K_\infty \Cho \opnorm{\mu}_{B^\star}\right)^m \frac{(t-r)^{m}}{m!} 
    \norm{\nabla\left(f\left(\Phi_{r,t}(x),\cdot\right)\right)}_{\sY-\infty} \\
+m\left(\frac32 K_\infty \Cho \opnorm{\mu}_{B^\star}\right)^m \frac{(t-r)^{m-1}}{(m-1)!} 
    \frac{\Chd t_0 e^{\opnorm{\nabla u}(t-r)}}{\Cho}\norm{f\left(\Phi_{r,t}(x),\cdot\right)}_{\sY-\infty}
\end{multline*}
and
\begin{multline*}
\norm{\nabla f_{r,t}^m}_{\Lpb^d}
\leq
\left(\frac32 K_\infty \Cho \opnorm{\mu}_{B^\star}\right)^m \frac{(t-r)^{m}}{m!} 
    e^{\opnorm{\nabla u}(t-r)}\norm{\nabla f}_{\Lpb^d}\II(t-r\geq t_0) \\
+m\left(\frac32 K_\infty \Cho \opnorm{\mu}_{B^\star}\right)^m \frac{(t-r)^{m-1}}{(m-1)!} 
    \frac{\Chd t_0 e^{\opnorm{\nabla u}(t-r)}}{\Cho}\norm{f}_{\Lpb}\II(t-r\geq t_0).
\end{multline*}
\end{prop}
\begin{proof}
The first inequality is established by induction making use of Proposition~\ref{p:zero} for the terms in $f$.
The second inequality introduces Proposition~\ref{p:dxPhi2} to get an estimate for $\nabla \Phi_{r,t}(x)$
\end{proof}

\begin{prop}\label{p:ApropD}
Let $T>0$, $0\leq r \leq t < T$, $\mu\in L^\infty\left([0,T),\left(\Mi,\norm{\cdot}_{B^\star}\right)\right)$
and either H3 hold or H2 hold but with $\mu$ bounded for all (not just Lebesgue almost all) times, then $A^{r,t}[\mu]$ is a propagator on \Wopb
and there is a $\CAW \in \RR$ such that
\begin{equation*}
\norm{A^{r,t}[\mu]}_{\Wopb\rightarrow \Wopb}
   \leq e^{\left(\opnorm{\nabla u}+ \frac32 K_\infty \Cho \opnorm{\mu}_{B^\star}\right) (t-r)}
     \CAW \II(t-r \leq t_0).
\end{equation*}
Further, for any $f \in D$ and almost all $t$
\begin{equation*}
\deriv{t}A^{r,t}[\mu]f=A^{r,t}[\mu]A_t[\mu]f,
\quad
\deriv{r}A^{r,t}[\mu]f=-A_r[\mu]A^{r,t}[\mu]f.
\end{equation*}
\end{prop}
\begin{proof}
From Proposition~\ref{p:Aderivterm} one sees that, for $f\in\Wopb$
\begin{multline}
\norm{\nabla\left(A^{r,t}[\mu]f\right)}_{\Lpb^d}
\leq
e^{\left(\frac32 K_\infty \Cho \opnorm{\mu}_{B^\star}+\opnorm{\nabla u}\right) (t-r)}
 \II(t-r \leq t_0)\times\\
\left( \norm{\nabla f}_{\Lpb^d}+
  \frac{3K_\infty \Chd t_0\opnorm{\mu}_{\Lpb^\star}}{2}
  \left(1+\frac32 K_\infty \Cho \opnorm{\mu}_{B^\star} (t-r)\right) \norm{ f}_{\Lpb}\right).
\end{multline}
For the $f$ part of the \Wopb-norm use Proposition~\ref{p:ApropB}, the first statement of that
proposition also established the boundary condition for \Wopb.  Differentiation in $r$ and $t$
is performed term by term in the infinite sum from Definition~\ref{d:Asprop}.
\end{proof}

These results concerning the dual propagators are concluded by showing
Lipschitz continuity in the measure valued path parameter.
\begin{prop}\label{p:lipsch}
Let $T >0$, suppose $\mu,\nu \in L^\infty\left([0,T),\left(\Mi,\norm{\cdot}_{B^\star}\right)\right)$,
$0\leq s\leq t < T$ and H1 holds, then
\begin{multline*}
\norm{A^{s,t}[\mu]-A^{s,t}[\nu]}_{\Lpb\rightarrow\Lpb}\\
\leq
\frac32 K_\infty \Cho e^{3 K_\infty \Cho \max(\opnorm{\mu}_{B^\star},\opnorm{\nu}_{B^\star})(t-s)}
  \II(t-s\leq t_0) \esssup_{r\in[s,t]} \norm{\mu_r - \nu_r}_{B^\star}\dd r.
\end{multline*}
and
\begin{multline*}
\norm{A^{s,t}[\mu]-A^{s,t}[\nu]}_{\Lpb\rightarrow\Lpb}\\
\leq
\frac32 K_\infty \Cho (t-s) e^{\frac32 K_\infty \Cho \max(\opnorm{\mu}_{B^\star},\opnorm{\nu}_{B^\star})(t-s)}
  \II(t-s\leq t_0) \int_s^t \norm{\mu_r - \nu_r}_{B^\star}\dd r.
\end{multline*}
\end{prop}
\begin{proof}
Write $M =  \max(\opnorm{\mu}_{B^\star},\opnorm{\nu}_{B^\star})$ and show by induction that
\begin{multline}
\norm{ U^{r,s_1}H_{s_1}[\mu] \dotsm U^{s_{m-1},s_m} H_{s_m}[\mu] U^{s_m,t}
          - U^{r,s_1}H_{s_1}[\nu] \dotsm U^{s_{m-1},s_m} H_{s_m}[\nu] U^{s_m,t}}_{\Lpb\rightarrow\Lpb}\\
\leq
\frac32 K_\infty \Cho \left(\frac32 K_\infty \Cho M\right)^{m-1}
  \sum_{j=1}^m \norm{\mu_{s_j} - \nu_{s_j}}_{B^\star}.
\end{multline}
\end{proof}
This result exploits a small amount of additional problem structure to adapt the method
set out in the proof of Theorem 2.12 in \citep{Kolok10b}.  The key is that the parameterisation
only affects the coagulation ($H$) part of the propagator, which has a bounded generator,
while the transport ($U$) part of the propagator, which has an unbounded generator is
independent of the parameterisation by $\mu$ and $\nu$.

\section{Operators on the Space of Measures}
Under the duality pairing of \Lpb and \Mi given by $\left<f,\mu\right>=\int_{\sX\times\sY}f\mu(\dd x,\dd y)$
as used above, (dual) operators $\Lpb\rightarrow\Lpb$ define (pre-dual) operators $\Mi \rightarrow \Mi$
with the same operator norms.  
\begin{defn}\label{d:dualops}
Let $0\leq s\leq t < T$ and $\mu \in L^\infty\left([0,T),\left(\Mi,\norm{\cdot}_{B^\star}\right)\right)$.
For the pre-duals of $U^{s,t}$ and $A^{s,t}[\mu]$ write
$\widetilde{U}^{t,s}$ and $\widetilde{A}^{t,s}[\mu]$ respectively and note the reversal of the time
indices.  For the pre-dual of $H_t[\mu]$ write $\widetilde{H}_t[\mu]$.
\end{defn}
It is emphasised that  $A^{s,t}[\mu]$ acts on functions while $\widetilde{A}^{t,s}[\mu]$ acts on measures,
but both are parameterised by a measure-valued path $\mu$.

The existence of the dual operators and their norm estimates is immediate, see for example
\citep[Thrm 2.10]{Kolok10b}.  The duality relations yield:
\begin{prop}\label{p:Adual}
Let $0\leq s\leq t < T$, $\mu,\nu \in L^\infty\left([0,T),\Mi\right)$ and assume H1 holds, then
\begin{equation*}
\norm{ \widetilde{A}^{t,s}[\mu]}_{\Mi \rightarrow \Mi}
=
\norm{A^{s,t}[\mu]}_{B \rightarrow B}
\leq
e^{\frac32 K_\infty \Cho \opnorm{\mu}_{B^\star}(t-s)}\II\left(t-s \leq t_0\right),
\end{equation*}
along with
\begin{multline*}
\norm{\widetilde{A}^{t,s}[\mu]-\widetilde{A}^{t,s}[\nu]}_{\Mi\rightarrow\Mi}=
\norm{A^{s,t}[\mu]-A^{s,t}[\nu]}_{\Lpb\rightarrow\Lpb}\\
\leq
\frac32 K_\infty \Cho e^{3 K_\infty \Cho \max(\opnorm{\mu}_{B^\star},\opnorm{\nu}_{B^\star})(t-s)}
  \II(t-s\leq t_0) \int_{r\in[s,t]} \norm{\mu_r - \nu_r}_{B^\star}\dd r.
\end{multline*}
and
\begin{multline*}
\norm{\widetilde{A}^{t,s}[\mu]-\widetilde{A}^{t,s}[\nu]}_{\Mi\rightarrow\Mi}=
\norm{A^{s,t}[\mu]-A^{s,t}[\nu]}_{\Lpb\rightarrow\Lpb}\\
\leq
\frac32 K_\infty \Cho (t-s) e^{\frac32 K_\infty \Cho \max(\opnorm{\mu}_{B^\star},\opnorm{\nu}_{B^\star})(t-s)}
  \II(t-s\leq t_0) \esssup_{r\in[s,t]}  \norm{\mu_r - \nu_r}_{B^\star}.
\end{multline*}
\end{prop}
\begin{proof}
Duality and Proposition~\ref{p:ApropB}.
\end{proof}

\begin{prop}\label{p:dualderiv}
Let $0\leq s\leq t < T$, $\mu \in L^\infty\left([0,T),\left(\Mi,\norm{\cdot}_{B^\star}\right)\right)$, $c \in \Mi$,
$f \in \Wopb$ and assume H1 holds, then for almost all $t$
\begin{equation*}
\deriv{t}\left<f,\widetilde{A}^{t,s}[\mu]c\right>=\left<A_t[\mu]f,\widetilde{A}^{t,s}[\mu]c\right>.
\end{equation*}
\end{prop}
\begin{proof}
Duality and Proposition~\ref{p:ApropB}
\end{proof}

\subsection{The Fixed Point Mapping}

This section presents a Picard iteration method for \eqref{eq:smol} highlighting the roles of the $B^\star$
and $D^\star$ norms on the space of measures.  The mapping that will be shown to have a fixed point is:
\begin{defn}\label{d:Psi}
Suppose $c_0 \in \Mi$, $0\leq  t < T$, let
$\mu \in L^\infty\left([0,T),\left(\Mi,\norm{\cdot}_{B^\star}\right)\right)$ and suppose H1 holds.  Define
$\Psi_{c_0}\colon L^\infty\left([0,T),\left(\Mi,\norm{\cdot}_{B^\star}\right)\right)\rightarrow L^\infty\left([0,T),\left(\Mi,\norm{\cdot}_{B^\star}\right)\right)$ by
\begin{equation}
\Psi_{c_0}(\mu)(t) = \widetilde{A}^{t,0}[\mu] c_0 +\int_0^t \widetilde{A}^{t,s}[\mu] I_s \dd s.
\end{equation}
\end{defn}

\begin{prop}\label{p:Psiprop}
Under the assumptions of Definition~\ref{d:Psi} one has
$\Psi_{c_0}(\mu) \in C_\mathrm{b}\left([0,T],\left(\Mi,\norm{\cdot}_{B^\star}\right)\right)$ with
\begin{equation*}
\norm{\Psi_{c_0}(\mu)(t)}_{B^\star}
\leq
e^{\frac32 K_\infty \Cho \opnorm{\mu}_{B^\star}\min(t, t_0)}
\left(\norm{c_0}_{B^\star} \II\left(t-s \leq t_0\right) + 
       \frac{2\sup_s \norm{I_s}_{B^\star}}{3 K_\infty \Cho \opnorm{\mu}_{B^\star}}\right).
\end{equation*}
The time derivative exists for almost all $t\in(0,T)$ with
\begin{equation}
\norm{\deriv{t}\Psi_{c_0}(\mu)(t)}_{\Wopb^\star} \leq
\norm{I_t}_{\Wopb^\star} +
 \left(\norm{u}_\infty +  \frac32 K_\infty \Cho\norm{\mu_t}_{B^\star}\right)\norm{\Psi_{c_0}(\mu)(t)}_{B^\star}
\end{equation}
and if $\mu \in C\left([0,T),\left(\Mi,\norm{\cdot}_{B^\star}\right)\right)$
then $\Psi_{c_0}(\mu)\in C^1\left((0,T),\left(\Mi,\norm{\cdot}_{D^\star}\right)\right).$
\end{prop}
\begin{proof}
$\Lpb^\star$ boundedness is a consequence of Proposition~\ref{p:Adual} and continuity follows from
the continuity in $t$ of $\widetilde{A}^{t,0}[\mu] $.

For the time derivative differentiate the formula in Definition~\ref{d:Psi}, and use Proposition~\ref{p:dualderiv}.
\end{proof}

\begin{prop}\label{p:fixedsoln}
Suppose $c_0 \in \Mi$, $T\in(0,\infty)$, H2 or H3 holds and
$c \colon [0,T) \rightarrow \left(\Mi,\norm{\cdot}_{B^\star}\right)$ is a bounded solution
to \eqref{eq:smol} with initial condition $c_0$, then $c$ is a fixed point of $\Psi_{c_0}$.
\end{prop}
\begin{proof}
Suppose $c$ to be a solution of \eqref{eq:smol} and let $t\in[0,T)$, then using duality and
Proposition~\ref{p:ApropD} one finds
\begin{equation}
\pderiv{r}\left<f,\widetilde{A}^{t,r}[c]c_r\right>=\left<f,\widetilde{A}^{t,r}[c]I_r\right>.
\end{equation}
Integrating over $r\in[0,t]$ completes the result.
This (standard) argument can be found, for example, in \citep[\S5.1]{Pazy83}.
\end{proof}

Proposition~\ref{p:fixedsoln} is the only place where one requires H2 or H3 in the existence and
uniqueness analysis.  This is in order to invoke Proposition~\ref{p:ApropD}
and more fundamentally so that $A^{s,t}[\mu]$
preserves \Wopb; otherwise one cannot give meaning to $\deriv{r}\widetilde{A}^{t,r}[\mu]$.  Without this
result it still follows that the mapping $\Psi$ has unique fixed point with all the advertised properties
(in particular solving \eqref{eq:smol}),
but one cannot rule out the possibility that there are additional (possibly less regular) solutions to
\eqref{eq:smol}.  These conclusions are stated more formally in Proposition~\ref{p:locexist} for which
two preparatory results are needed.

\begin{prop}\label{p:contrcontain}
Let $c_0 \in \Mi$, $M\in \RRP$ be large enough to satisfy
\begin{equation*}
M > \norm{c_0}_{B^\star} + \frac{2\sup_s \norm{I_s}_{B^\star}}{3 K_\infty \Cho M},
\end{equation*}
define $E_M=\left\{\mu \in \Mi \colon \norm{\mu}_{B^\star} \leq M \right\}$ and assume H1 holds.
Then there exists a $\tau_M>0$ such that $\Psi_{c_0}$ preserves $L^\infty\left([0,\tau_M),\left(E_M,\norm{\cdot}_{B^\star}\right)\right)$.
\end{prop}
\begin{proof}
Let $r_M> 1$ be given by
\begin{equation}
r_M \left(\norm{c_0}_{B^\star} + \frac{2\sup_s \norm{I_s}_{B^\star}}{3 K_\infty \Cho M}\right) = M
\end{equation}
and suppose $\mu \in L^\infty\left([0,T),\left(E_M,\norm{\cdot}_{B^\star}\right)\right)$ for some $T>0$.
Use Definition~\ref{d:Psi} along with the operator norm estimate
from Proposition~\ref{p:Adual} to see that, for $t < T$
\begin{equation}
\norm{\Psi_{c_0}(\mu)(t)}_{B^\star} \\
\leq
e^{\frac32 K_\infty \Cho M\min(t,t_0)} \frac{M}{r_M} 
\end{equation}
and so $\norm{\Psi_{c_0}(\mu)(t)}_{B^\star} \leq M$ if
$\min(t,t_0) \leq \frac{2\log r_M}{3K_\infty \Cho M}$.
Hence it is sufficient to take $\tau_M = \frac{2\log r_M}{3K_\infty \Cho M}$ and if $t_0$, the maximum residence
time for a particle, satisfies 
$t_0 \leq \frac{2\log r_M}{3K_\infty \Cho M}$ then one may take $\tau_M = \infty$.
\end{proof}

\begin{prop}\label{p:locfix}
Let $c_0 \in \Mi$ and $E_M,\tau_M$ be as in Proposition~\ref{p:contrcontain} and assume H1 holds,
then there is a $\tau_M^\prime \leq \tau_M$ such that $\Psi_{c_0}$ is a contraction on
$L^\infty\left([0,\tau_M^\prime),\left(E_M, \norm{\cdot}_{B^\star}\right)\right)$.
\end{prop}
\begin{proof}

Suppose $\mu$ and $\nu$ are in $L^\infty\left([0,\tau_M),\left(E_M, \norm{\cdot}_{B^\star}\right)\right)$
$f \in \Lpb$ and $t\in[0,\tau_M)$, then by Proposition~\ref{p:lipsch}
\begin{multline}
\norm{\Psi_{c_0}(\mu)(t) -\Psi_{c_0}(\nu)(t)}_{\Lpb^\star} \\
\leq
   \norm{\widetilde{A}^{t,0}[\mu]-\widetilde{A}^{t,0}[\nu]}_{B\rightarrow B}
   \norm{c_0}_{B^\star}\II\left(t\leq t_0\right)\\
 +\int_0^t    \norm{\widetilde{A}^{t,s}[\mu]-\widetilde{A}^{t,s}[\nu]}_{B\rightarrow B}
   \norm{I_s}_{B^\star}\II\left(t-s\leq t_0\right)\dd s\\
\leq
\frac32 K_\infty \Cho t e^{\frac32 K_\infty \Cho Mt} \norm{c_0}_{\Lpb^\star}
  \II(t\leq t_0) \esssup_{r\in[0,t_M)}  \norm{\mu_r - \nu_r}_{B^\star}\\
+\frac34 K_\infty \Cho \min(t^2,t_0^2) e^{\frac32 K_\infty \Cho M\min(t,t_0)}
   \sup_r \norm{I_r}_{\Lpb^\star} \esssup_{r\in[0,t_M)} \norm{\mu_r - \nu_r}_{B^\star}.
\end{multline}
Hence for any $0<r<1$ one can find a $\tau_M^\prime \leq \tau_M$ such that 
\begin{equation}
\sup_{t\in[0,\tau_M^\prime)}\norm{\Psi_{c_0}(\mu^1)(t) -\Psi_{c_0}(\mu^2)(t)}_{\Lpb^\star} \leq
r\esssup_{t\in[0,\tau_M^\prime)}\norm{\mu^1_t - \mu^2_t}_{\Lpb^\star}.
\end{equation}
\end{proof}

\begin{prop}\label{p:locexist}
Let $c_0 \in \Mi$ and $E_M$ be as in Proposition~\ref{p:contrcontain}, $\tau_M^\prime$ as in
Proposition~\ref{p:locfix} and assume H1 holds, then \eqref{eq:smol} with initial condition $c_0$
has a solution on $[0,\tau_M^\prime)$ and this solution is in
$C_\mathrm{b}\left([0,\tau_M^\prime),\left(E_M, \norm{\cdot}_{B^\star}\right) \right)$.
If H2 or H3 hold this solution is unique.
\end{prop}
\begin{proof}
By Proposition~\ref{p:locfix} there is precisely one fixed point of $\Psi_{c_0}$, which by
Proposition~\ref{p:Psiprop} is a solution of \eqref{eq:smol} with initial condition $c_0$.
Proposition~\ref{p:contrcontain} shows
that this solution is in $C_\mathrm{b}\left([0,t_M^\prime),\left(E_M,\norm{\cdot}_{B^\star}\right)\right)$.
By Proposition~\ref{p:fixedsoln} every solution of \eqref{eq:smol} with initial condition $c_0$
is a fixed point of $\Psi_{c_0}$ and thus is unique.  
\end{proof}

\begin{prop}\label{p:globuniq}
Let $T>0$, assume H1 and suppose $\Psi_{\mu_0}$ and $\Psi_{\nu_0}$ have fixed points
$\mu$ and $\nu$ respectively.  Write $M = \max\left(\opnorm{\mu}_{B^\star}, \opnorm{\nu}_{B^\star}\right)$,
then there exists $\CGro(M) > 0$ such that for $t \leq T$
\begin{equation*}
\norm{\mu_t-\nu_t}_{\Lpb^\star}
\leq
\norm{\mu_0-\nu_0}_{\Lpb^\star} e^{\frac32 K_\infty \Cho M \min(t,t_0)} e^{\CGro(M)t}
\end{equation*}
and thus at most one finite solution is possible for any given initial condition.
\end{prop}
\begin{proof}
Since any solution must be a fixed point of $\Psi$ for the appropriate initial condition
\begin{multline}
\norm{\mu_t-\nu_t}_{\Lpb^\star} 
=
\norm{\Psi_{\mu_0}(\mu)(t)-\Psi_{\nu_0}(\nu)(t)}_{\Lpb^\star} \\
\leq 
\norm{\Psi_{\mu_0}(\mu)(t)-\Psi_{\mu_0}(\nu)(t)}_{\Lpb^\star}
+\norm{\Psi_{\mu_0}(\nu)(t)-\Psi_{\nu_0}(\nu)(t)}_{\Lpb^\star}.
\end{multline}
Now by Proposition~\ref{p:Adual} estimate the second term as follows
\begin{multline}
\norm{\Psi_{\mu_0}(\nu)(t)-\Psi_{\nu_0}(\nu)(t)}_{\Lpb^\star}
=
\norm{\widetilde{A}^{t,0}[\nu]\left(\mu_0-\nu_0\right)}_{\Lpb^\star}
\leq
e^{\frac32 K_\infty \Cho M t} \II\left(t\leq t_0\right) \norm{\mu_0-\nu_0}_{\Lpb^\star}.
\end{multline}
For the first term using Proposition~\ref{p:Adual} one finds
\begin{multline}
\norm{\Psi_{\mu_0}(\mu)(t) -\Psi_{\mu_0}(\nu)(t)}_{\Lpb^\star} \\
\leq
   \norm{\widetilde{A}^{t,0}[\mu]-\widetilde{A}^{t,0}[\nu]}_{B\rightarrow B}
   \norm{\mu_0}_{B^\star}\II\left(t\leq t_0\right)\\
 +\int_0^t    \norm{\widetilde{A}^{t,s}[\mu]-\widetilde{A}^{t,s}[\nu]}_{B\rightarrow B}
   \norm{I_s}_{B^\star}\II\left(t-s\leq t_0\right)\dd s\\
\leq
\frac32 K_\infty \Cho e^{3 K_\infty \Cho Mt} \norm{\mu_0}_{\Lpb^\star}
  \II(t\leq t_0) \int_{r\in[0,T)}  \norm{\mu_r - \nu_r}_{B^\star}\dd r \\
+\frac32 K_\infty \Cho e^{3 K_\infty \Cho M\min(t,t_0)} \min(t,t_0)
   \sup_r \norm{I_r}_{\Lpb^\star} \int_{r\in[0,T)} \norm{\mu_r - \nu_r}_{B^\star}.
\end{multline}
so using Gronwall with
\begin{equation}
\CGro(M)=
\frac32K_\infty \Cho e^{\frac32K_\infty \Cho M t_0}
\left(\norm{c_0^1}_{\Lpb^\star}+ t_0  \sup_r \norm{I_r}_{\Lpb^\star}\right)
\end{equation}
one has
\begin{equation}
\norm{\mu_t-\nu_t}_{\Lpb^\star}
\leq
\norm{\mu_0-\nu_0}_{\Lpb^\star} e^{\frac32 K_\infty \Cho M \min(t,t_0)} e^{\CGro(M)t}.
\end{equation}

\end{proof}

\begin{proof}[Proof of Theorem \ref{t:locexist}]
The existence of a solution on a small time interval is the conclusion of Proposition~\ref{p:locexist},
this procedure may be iterated, but the time steps may decay so that a solution cannot necessarily
be constructed for all time.

Proposition~\ref{p:fixedsoln} establishes a representation for any solutions, should they exist.  Using
this representation boundedness and continuity in the $B^\star$-norm along with differentiability 
in the $\Wopb^\star$-norm were established in Proposition~\ref{p:Psiprop}.

For compact subsets of the time interval on which a solution exists (which may be longer than
the time interval for which this theorem proves existence), $\Lpb^\star$ Lipschitz continuity in the
initial conditions and uniqueness are consequences of Proposition~\ref{p:globuniq}. 
\end{proof}

\subsection{Positive Measures}
Write $\Lpb^+$ for the cone of non-negative functions in \Lpb and $\Mi^+,\Moi^+$ for the cone of
non-negative measures in \Mi, respectively \Moi.  These cones are of course not Banach spaces,
but one would expect the physical solutions of any reaction--transport problem to remain in
$\Mi^+$, if they start there.  This is indeed the case and turns out to allow the local existence result for the
coagulation--transport problem studied here to be extended to a global one, which along with the
results already established makes the problem well posed.

\begin{prop}\label{p:pospres}
Let $T>0$ and $\mu\in L^\infty\left([0,T),\left(\Mi^+,\norm{\cdot}_{B^\star}\right)\right)$, then for
$0\leq s\leq t <T$ $A^{s,t}[\mu]$ is a positivity preserving on \Lpb, the same is true of
$\widetilde{A}^{t,s}[\mu]$ on \Mi and both operators are contractions on the respective positive
cones, that is
\begin{equation*}
\normoi{A^{s,t}[\mu]f}  \leq \normoi{f},\ f\in B^+, \qquad
\norm{\widetilde{A}^{t,s}[\mu]\nu}_{\Lpb^\star} \leq \norm{\nu}_{\Lpb^\star},\ \nu\in\Mi^+.
\end{equation*}

\end{prop}
\begin{proof}
A proof for the dual propagators on \Lpb suffices.  For this note that $U^{s,t}$ is positivity preserving
with \Lpb-operator norm 1.  One further checks that $H_t[\mu]$ generates a positivity preserving
propagator with operator norm at most 1 on $\Lpb^+$, which will be denoted $H^{s,t}[\mu]$.
One can now approximate $A^{s,t}[\mu]$ by
\begin{equation}
U^{t,t_{m-1}}H^{t,t_{m-1}}[\mu] \dotsm U^{t_1,t_2}H^{t_1,t_2}[\mu] U^{s,t_1}H^{s,t_1}[\mu],
\quad
t_i = s + i\frac{t-s}{m},\ i=1,\dotsc m-1,\ m\in\NN
\end{equation}
which is a splitting, to see positivity is preserved and the operator norm is bounded above by 1.
\end{proof}

The key estimate from Proposition~\ref{p:Psiprop} can now be improved (recall $t_0$ is the
maximum particle residence time from \S\ref{s:flowass}):
\begin{prop}\label{p:Psiproppos}
Assume H1 holds, $c_0 \in \Mi^+$ and $\mu \in L^\infty\left([0,T),\Mi^+\right)$ for
$T\in[0,\infty)$ then
\begin{equation*}
\norm{\Psi_{c_0}(\mu)(t)}_{B^\star}
\leq
\norm{c_0}_{B^\star} \II\left(t\leq t_0\right) + \min(t,t_0) \sup_s\norm{I_s}_{B^\star}.
\end{equation*}
\end{prop}
\begin{proof}
This follows from Definition~\ref{d:Psi}, and the norm estimates in Proposition~\ref{p:pospres}.
\end{proof}

\begin{proof}[Proof of Theorem \ref{t:globexist}]
One can take $M=\norm{c_0}_{B^\star}+ t_0 \sup_s\norm{I_s}_{B^\star}$ and $t_M = \infty$
in Proposition~\ref{p:contrcontain}.  Proposition~\ref{p:locfix}
then extends to show that $\Psi_{c_0}$ is a contraction on
$L^\infty\left([0,\infty), E_M\cap \Mi^+\right)$.
\end{proof}

\subsection{Measures with Lebesgue Densities}\label{s:dens}
One would of course like to prove that every measure valued to solution to \eqref{eq:smol}
is in fact also a strong solution to an appropriate extension of \eqref{e:smolold}.  The main difficulty
that has to be addressed in this section is the inflow of pre-existing particles through
$\Gamma_\mathrm{in}$  which leads to $I_t$ having a singular (with respect to Lebesgue measure
on \sX) part concentrated on $\Gamma_\mathrm{in}$.  In this section it is
shown that under a mild time-regularity condition (I2) the advective transport smooths out
the inception concentrated on $\Gamma_\mathrm{in}$ sufficiently for solutions to \eqref{eq:smol}
to remain in \Moi.  Shocks are of course preserved by advective transport, but what happens here
is more like spraying paint onto a moving surface, as long as the surface keeps moving a thin layer of paint
is deposited everywhere and no ridge (shock) is created. 

\begin{prop}\label{p:Adualrep}
Assume H1 holds, $0\leq s\leq t < T$, $\mu \in L^\infty\left([0,T),\left(\Mi,\norm{\cdot}_{B^\star}\right)\right)$, then
\begin{equation*}
\widetilde{A}^{t,s}[\mu] = \widetilde{U}^{t,s} + \sum_{m=1}^\infty
\int_{r\leq s_1 \leq \dotsc \leq s_m\leq t}
 \widetilde{U}^{t,s_m}\widetilde{H}_{s_m}[\mu] \widetilde{U}^{s_m,s_{m-1}} \dotsm 
  \widetilde{H}_{s_1}[\mu] \widetilde{U}^{r,s_1}
 \dd s_1 \dotsm \dd s_m.
\end{equation*}
\end{prop}
\begin{proof}
For each $m$ the term in the sum here is dual to the term with the same $m$ in Definition~\ref{d:Asprop}.
\end{proof}

\begin{prop}\label{p:dualseries}
Assume H1 holds, $0\leq s\leq t < T$, $\mu \in L^\infty\left([0,T),\left(\Mi,\norm{\cdot}_{B^\star}\right)\right)$ and $c\in\Moi$, then for any $\phi \in \mathcal{B}_\mathbf{b}(\sY)$ and bounded measurable $f\colon \sX\rightarrow\RR$
\begin{equation*}
\int_\sX f(x)\int_\sY \phi(y)\widetilde{U}^{t,s}c(x,\dd y)\dd x
=\int_\sX f(x)e^{-\int_s^t \nabla \cdot u_r \left(\Phi_{t,r}(x)\right)\dd r}
  \int_\sY \phi(y)c\left(\Phi_{t,s}(x),\dd y\right)\dd x
\end{equation*}
and
\begin{multline*}
\int_\sY \phi(y)\widetilde{H}_t[\mu]c(x,\dd y)
=
\int_\sY \frac12 \int_{\sX\times\sY}\phi(y+y_2)h(x,x_2)K(y,y_2)\mu_t(\dd x_2,\dd y_2) c(x,\dd y)\\
-\int_\sY \frac12 \int_{\sX\times\sY}\phi(y)\left[h(x,x_2)+h(x_2,x)\right]K(y,y_2)\mu_t(\dd x_2,\dd y_2) c(x,\dd y).
\end{multline*}
\end{prop}
\begin{proof}
For the first statement, which concerns the transport propagator $U$, one makes the change of variable
$x \leftrightarrow \Phi_{s,t}(x)$.  Liouville's formula then gives the determinant of the Jacobian as
$\abs{\det \frac{\partial \Phi_{s,t}(x)}{\partial x}} = \exp{\int_s^t \nabla \cdot u_t\left(\Phi_{t,r}(x)\right)\dd r}$.
Alternatively one can approximate $c$ by \sX-differentiable functions (since the claim is only of an $L^1$
nature) and check the formula directly using Proposition~\ref{p:dsPhi}.

For $H_t$ use Definition~\ref{d:Htstar}; the important point is that the new measure also has a
density with respect to Lebesgue measure on \sX.
\end{proof}

\begin{prop}\label{p:densprop}
Assume H1 holds, $0\leq s\leq t < T$, $\mu \in L^\infty\left([0,T),\left(\Mi,\norm{\cdot}_{B^\star}\right)\right)$ then $\widetilde{A}^{t,s}[\mu]$ is
a bounded propagator on \Moi with
\begin{equation*}
\norm{\widetilde{A}^{t,s}[\mu]}_{\Moi \rightarrow \Moi}
\leq
e^{\left(\norm{\nabla\cdot u}_\infty +\frac32 K_\infty \Cho \opnorm{\mu}_{B^\star}\right) (t-s)}\II(t-s\leq t_0).
\end{equation*}
\end{prop}
\begin{proof}
From Proposition~\ref{p:dualseries} one sees that
\begin{equation*}
\norm{\widetilde{U}^{t,s}}_{\Moi \rightarrow \Moi}
\leq
e^{\norm{\nabla\cdot u}_\infty (t-s)}
\end{equation*}
and
\begin{equation*}
\norm{\widetilde{H}_t[\mu]}_{\Moi \rightarrow \Moi}
\leq
\frac32 K_\infty \Cho \opnorm{\mu}_{B^\star}.
\end{equation*}
The proof now follows that of Proposition~\ref{p:ApropB}.
\end{proof}
The $\widetilde{U}$ and therefore also the $\widetilde{A}$ are not (norm-)continuous on \Moi.  This can easily
be seen by considering a small translation of a step function regarded as the density of a measure in \Moi.
The eventual time continuity of the solutions will depend on having some \sX-regularity for the densities of
the measures in \Moi.

The next proposition provides a better norm estimate when the propagator is restricted to positive
measures.  This is then used in Propositions~\ref{p:bdrydens2}\&\ref{p:bdrydens} to show that
inception concentrated on the inflow boundary does not take the solution out of \Moi.
\begin{prop}\label{p:posdensprop}
Assume H1 holds, $0\leq s\leq t < T$, $\mu \in L^\infty\left([0,T),\left(\Mi^+,\norm{\cdot}_{B^\star}\right)\right)$ and let $c \in \Moi$ be a positive measure, then
$\widetilde{A}^{t,s}[\mu] c$ is also a positive measure and
\begin{equation*}
 \norm{\widetilde{A}^{t,s}[\mu]c}_{\Moi} \leq
 e^{\norm{\nabla\cdot u}_\infty (t-s)}\II(t-s\leq t_0) \norm{c}_{\Moi}.
\end{equation*}
\end{prop}
\begin{proof}
Since $\Moi \subset \Mi$ preservation of positivity is a consequence of Proposition~\ref{p:pospres}
and Proposition~\ref{p:dualseries} states the \Moi is preserved.  To proceed note that the propagator
generated by $\widetilde{H}_t[\mu]$ preserves $\Moi^+$ not just \Moi and coagulation reduces
$c(x,\sY)=\norm{c(x,\cdot)}_{\sY-\mathrm{TV}}$ for all $x\in\sX$.  Secondly $\widetilde{U}^{t,s}$ is
positivity preserving and $\norm{\widetilde{U}^{t,s}}_{\Moi} \leq e^{\norm{\nabla \cdot u}(t-s)}$
using the representation from Proposition~\ref{p:dualseries} so the result now follows by the same
splitting approximation as in the proof of Proposition~\ref{p:pospres}.

\end{proof}

\begin{prop}\label{p:bdrydens2}
Let $T>0$, and $\nu\colon [0,T)\rightarrow \Mi$ be such that (note the reduction in the domain of integration
accompanied by a change in the position of the time argument)
\begin{equation*}
\int_{\sX\times\sY}f(x,y) \nu_t(\dd x, \dd y) =
\int_{\Gamma_\mathrm{in}\times\sY}f(\xi,y) \nu(t,\xi, \dd y)\dd \xi \qquad \forall f \in\Lpb \quad \forall t\in[0,T).
\end{equation*}
Suppose further that there is a $\nu_\ast\in(0,\infty)$ such that
$\sup_{t\in [0,T),\xi\in\Gamma_\mathrm{in}}\norm{\nu(t,\xi,\cdot)}_{\sY-\mathrm{TV}}/u_t(\xi)\cdot n(\xi)\leq\nu_\ast$, then
$\int_0^t \widetilde{U}^{t,s}\nu_s \dd s \in L^\infty\left([0,T), \Moi\right)$ and for all (not just almost all) $t<T$
\begin{equation*}
\norm{\int_0^t \widetilde{U}^{t,s}\nu_s \dd s}_{\Moi}
\leq
\nu_\ast e^{\norm{\nabla\cdot u}_\infty \min(t,t_0)}.
\end{equation*}
\end{prop}
\begin{proof}
Let $\xi\in \Gamma_\mathrm{in}$ and take an orthonormal basis for $\RR^d$ at $\xi$ given by
$e_1 = n(\xi)$ the outward normal and $e_2, \dotsc, e_d \in \Gamma_\mathrm{in}$.  With respect
to this basis let the rows of the matrix $\nabla \Phi_{r,t}(x)\mid_{x=\xi}$ be $\partial_i \Phi_{r,t}(x)\mid_{x=\xi}$.
Thus rows $2,\dotsc, d$ of this matrix are the same as rows $2,\dotsc, d$ of
$\frac{\partial \Phi_{r,t}(\xi)}{\partial (r,\xi)}$ and using Proposition~\ref{p:dsPhi} the first row is
\begin{multline}
\pderiv{r} \Phi_{r,t}(\xi) = -\nabla \Phi_{r,t}(x)\mid_{x=\xi} u_r(\xi) =\\
-\sum_{i=1}^d \nabla \Phi_{r,t}(x)\mid_{x=\xi} e_i \left(e_i \cdot u_r(\xi)\right)=
-\sum_{i=1}^d \partial_i \Phi_{r,t}(x)\mid_{x=\xi}  \left(e_i \cdot u_r(\xi)\right),
\end{multline}
which is $\pm\left(\partial_1 \Phi_{r,t}(x)\mid_{x=\xi} \right) \left(n(\xi) \cdot u_r(\xi)\right)$ plus a linear
combination of the remaining rows.  One thus has for $\xi\in\Gamma_\mathrm{in}$
\begin{equation}
\det \left(\frac{\partial \Phi_{r,t}(\xi)}{\partial (r,\xi)}\right) =
-u_r(\xi) \cdot n(\xi) \det\left(\nabla \Phi_{r,t}(x)\right)  \biggr\rvert_{x=\xi}.
\end{equation}

Now let $f\in\Lpb$ with $f\left(\Phi_{r,t}(\xi),y\right)=0$ for $\Phi_{r,t}(\xi) \notin \sX$
as in the definition of $U^{r,t}$ so
\begin{multline}
\left<f,\int_s^t \widetilde{U}^{t,r}\nu_r \dd r\right>=
\int_s^t\int_{\Gamma_\mathrm{in}}\int_\sY f\left(\Phi_{r,t}(\xi),y\right)\nu(r,\xi,\dd y)\dd \xi\dd r\\
=\int_{\substack{x\colon x=\Phi_{r,t}(\xi) \\ r\in(s,t),\xi\in\Gamma_\mathrm{in}}}
  \det \left(\frac{\partial \Phi_{r,t}(\xi)}{\partial (r,\xi)}\right)^{-1}\int_\sY f\left(x,y\right){\nu}(r,\xi,\dd y)\dd x\\
=\int_{\substack{x\colon x=\Phi_{r,t}(\xi) \\ r\in(s,t),\xi\in\Gamma_\mathrm{in}}}
  \abs{\det \left(\nabla \Phi_{r,t}(\xi)\right)^{-1}}\norm{f\left(x,\cdot\right)}_{\sY-\infty}
  \abs{u_r(\xi)\cdot n(\xi)}^{-1}\norm{\widehat{\nu}(x,\cdot)}\dd x\\
\leq e^{\norm{\nabla\cdot u}_\infty \min(t-s,t_0)}\nu_\ast
\int_{\substack{x\colon x=\Phi_{r,t}(\xi) \\ r\in(s,t),\xi\in\Gamma_\mathrm{in}}}
  \norm{f\left(x,\cdot\right)}_{\sY-\infty} \dd x,
\end{multline}
where $\widehat{\nu}(x,\dd y)$ is defined to be $\nu(r,\xi,\dd y)$ for the unique $r,\xi$ such that
$\Phi_{r,t}(\xi)=x$.  Proposition~\ref{p:dxPhi2} in the Appendix provides the estimate for the determinant.
\end{proof}

\begin{prop}\label{p:bdrydens}
Let $T>0$ and $\mu\in L^\infty\left([0,T),\left(\Mi^+,\norm{\cdot}_{B^\star}\right)\right)$, then under the
conditions of Proposition~\ref{p:bdrydens2}
\begin{equation*}
\norm{\int_0^t \widetilde{A}^{t,s}[\mu]\nu_s \dd s}_{\Moi\rightarrow\Moi}
\leq
\nu_\ast e^{\norm{\nabla\cdot u}_\infty \min(t,t_0)}.
\end{equation*}
\end{prop}
\begin{proof}
Use the series expansion from Proposition~\ref{p:Adualrep}
and the \Moi-operator norm estimates from Proposition~\ref{p:posdensprop}.
\end{proof}

\begin{proof}[Proof of Theorem \ref{t:globdens}]
Theorem~\ref{t:globexist} provides the existence of a solution $c$.  Proposition~\ref{p:fixedsoln} shows that this solution satisfies
\begin{equation}\label{e:measrep}
c_t = \widetilde{A}^{t,0}[c] c_0 +\int_0^t \widetilde{A}^{t,s}[c] I_s \dd s \quad t\geq 0.
\end{equation}
Propositions~\ref{p:posdensprop}\&\ref{p:bdrydens} show that this is in \Moi for all times.  The boundedness follows from the estimates in the same two propositions.
\end{proof}

In order to obtain a strong solution to \eqref{eq:smol} it is not sufficient that the measure valued solutions
have a density with respect to Lebesgue measure on \sX, this density should itself have a derivative.

\subsection{Differentiability}

One could proceed as in Proposition~\ref{p:ApropD} to see that $\widetilde{A}^{t,s}[c]$ preserves
measures with \sX-differentiable densities except for possible jumps where $s,t,x$ are such that
$\Phi_{t,s}(x) \in \Gamma_\mathrm{in}$.  This leaves two questions open---how to handle these
jumps and secondly the treatment of the integral term from \eqref{e:measrep} and in particular the
$I_\mathrm{bdry}$ part of $I$ in that integral.  The right approach to these tasks seems to be to introduce
the space of measures with \sX-bounded variation densities:
\begin{defn}\label{d:bv}
\begin{equation*}
\Mbv=\left\{c\in\Moi \colon \left(\exists C=C(c)\right)\left(\forall f \in D^d\right)\left(\left<\nabla\cdot f,c\right> 
\leq C\norm{f}_{B^d}\right)\right\}.
\end{equation*}
\end{defn}
This is equivalent to the existence of a measure $\nabla c \in \Mi^d$ (not necessarily in \Moi) such
that $\left<\nabla f, c\right>=-\left<f, \nabla c\right>$.

Until now the notation $\left<f,\mu\right>$ has been used for $\int_{\sX\times\sY}f(x,y)\mu(\dd x,\dd y)$
for $f\in\Lpb$ and $\mu \in \Mi$.  To consider derivatives it is necessary to move to vector valued functions
and measures; to facilitate this the notation is extended so that for $g\in\Lpb^d$ and $\nu \in \Mi^d$
\begin{equation}
\left<g,\nu\right>:=\sum_{i=1}^d\int_{\sX\times\sY}g_i(x,y)\nu_i(\dd x,\dd y).
\end{equation}

Some more definitions are now needed for the proof that $\int_0^t \widetilde{A}^{t,s}I_s\dd s$ and by
extension the entire solution is \Mbv.  First recall $s(t,x)$ from \S\ref{s:flowass}, the time at which a particle
travelling with the flow must have entered the domain in order to reach $x$ at time $t$.  Since $t$
is fixed in the relevant places $s(x)$ will be written for brevity in numerous sub- and superscripts, the
$t$ should be understood.
\begin{defn}\label{d:ftilde}
Let $f\in\Lpb$, $t>0$ and $\mu\in C\left([0,t],\Mi\right)$.  Define $f^0_{r,t}=U^{r,t}f$,
and $f^{m+1}_{r,t}=\int_r^t U^{r,s}H_s[\mu]f^m_{s,t}\dd s$ as in Definition~\ref{d:Aterm}.  The define $\tilde{f}$
by
\begin{equation*}
\tilde{f}^m_{r,t}(x,y) = f^m_{r,t}\left(\Phi_{t,r}(x),y\right)\II\left(\Phi_{t,r}(x)\in\sX\right)
\end{equation*}
the operators $S^{r,t}\colon \Lpb\rightarrow\Lpb$ by
\begin{equation*}
S^{r,t}f(x,y) = \sum_{m=0}^\infty\tilde{f}^m_{r,t}(x,y)
\end{equation*}
and finally the operator $\mathcal{S}^{t}\colon \Lpb\rightarrow\Lpb$ by
\begin{equation*}
\mathcal{S}^{t}f(x,y) = S^{s(x),t}f(x,y) = \sum_{m=0}^\infty\tilde{f}^m_{s(x),t}(x,y).
\end{equation*}
This operator can also be regarded as acting on $\Lpb^d$ by applying it componentwise.
\end{defn}

\begin{prop}\label{p:Snorm}
Let $f\in\Lpb$ or $f\in\Lpb^d$, $t>0$ and $\mu\in C\left([0,t],\Mi\right)$, then
\begin{equation*}
\norm{\mathcal{G}^t f(x,\cdot)}_{\sY-\infty} \leq
e^{\frac32 K_\infty\Cho \opnorm{\mu}_{\Lpb^\star}\left(t-s(x)\right)}\norm{f(x,\cdot)}_{\sY-\infty} 
\end{equation*}
for all $x\in\sX$ and the operator even preserves $D$.
\end{prop}
\begin{proof}
This is an exercise in estimating the terms of the summation as in Propositions~\ref{p:zero}\&\ref{p:Aderivterm}.
\end{proof}

\begin{prop}\label{p:divS}
Let $F\in\Wopb^d$, $t>0$ and $\mu\in C\left([0,t],\Mi\right)$, then the operator
$\left(\nabla\cdot\mathcal{S}^t\right) \colon \Wopb^d\rightarrow \Lpb$
defined by
\begin{equation*}
\left(\nabla\cdot\mathcal{S}^t\right)F =
\nabla\cdot\left(\mathcal{S}^t F\right)-\mathcal{S}^t \left(\nabla\cdot F\right)
\end{equation*}
satisfies
\begin{multline*}
\norm{\left(\nabla\cdot\mathcal{S}^t\right)F(x,\cdot)}_{\sY-\infty}
\leq
\norm{F(x,\cdot)}_{\sY-\infty}\times \\
  \left(3K_\infty\Chd t_0 \opnorm{\mu}_{\Lpb^\star}+\norm{\nabla s}_\infty \right)
  \left(1+\frac32 K_\infty\Cho \opnorm{\mu}_{\Lpb^\star}\left(t-s(x)\right)\right)
  e^{\frac32 K_\infty\Cho \opnorm{\mu}_{\Lpb^\star}\left(t-s(x)\right)}
\end{multline*}
for all $x\in\sX$.  Of course $\mathcal{S}^t$ depends on $\mu$, but since this will always be the unique
solution to \eqref{eq:smol} this detail is ignored in the notation.
\end{prop}
\begin{proof}
Define $\tilde{F}^m_{r,t}$ by replacing $f$ with $F$ throughout Definition~\ref{d:ftilde} and let 
$f=\nabla \cdot F$ and let $\tilde{f}^m_{r,t}$ be as in Definition~\ref{d:ftilde}.  By induction one establishes
\begin{multline}
\norm{\nabla\cdot\tilde{F}^m_{r,t}(x,\cdot)-\tilde{f}^m_{r,t}(x,\cdot)}_{\sY-\infty}
\leq
\norm{F(x,\cdot)}_{\sY-\infty}\times \\
  \frac32 K_\infty\Chd t_0 \opnorm{\mu}_{\Lpb^\star}
  m\left(\frac32 K_\infty\Chd t_0 \opnorm{\mu}_{\Lpb^\star}\right)^{m-1}\frac{\left(t-s(x)\right)^{m-1}}{(m-1)!}.
\end{multline}
One then establishes a similar formula with $r$ replaced by $s(x)$ and the result follows.
\end{proof}

\begin{prop}\label{p:solnbv}
Assume H2 or H3 holds, that $c_0$ is in the positive cone of \Mbv, and that I3 is satisfied,
then the unique solution $c$ to \eqref{eq:smol} given by Theorem~\ref{t:globdens} satisfies
$c_t \in \Mbv\ \forall t\in\RRP$.
\end{prop}
\begin{proof}
Following \eqref{e:measrep} and \S\ref{s:incep} $c_t$ can be decomposed as
\begin{equation}
c_t = \widetilde{A}^{t,0}[c]c_0 + \int_0^t \widetilde{A}^{t,s}[c] I_\mathrm{int}(s) \dd s
  +\int_0^t \widetilde{A}^{t,s}[c] I_\mathrm{bdry}(s) \dd s.
\end{equation}
One checks in the same way as for the dual propagator in Proposition~\ref{p:Aderivterm} that
$\widetilde{A}^{t,s}[c]$ preserves \Mbv, possibly introducing a new jump on the manifold
$\left\{x\in\sX\colon \Phi_{t,s}(x)\in\Gamma_\mathrm{in}\right\}$.  This deals with the first two
terms in the above representation; the third term is somewhat more challenging.

Let $F\in D^d$ and $t\in \RRP$, then it is sufficient to show that 
\begin{equation}
\int_0^t \left<\nabla\cdot F,\widetilde{A}^{t,s}[c]I_\mathrm{bdry}(s)\right>\dd s =
\int_0^t \int_{\Gamma_\mathrm{in}} \int_\sY \left(A^{s,t}[c]f\right)(\xi,y) I_\mathrm{bdry}(s,\xi,\dd y)\dd \xi\dd s
\end{equation}
is bounded by a constant times $\norm{F}_{\Lpb^d}$.
One can introduce a change of variables $(s,\xi) \leftrightarrow x$ where ($t$ is fixed) $\Phi_{s,t}(\xi)=x$,
so $\xi = \Phi_{t,s}(x)$ is the point where fluid reaching $x$ at time $t$ entered the domain and the time of entry was $s$.
The determinant of the Jacobian for this transformation is the inverse of
\begin{equation}
\det \frac{\partial \Phi_{s,t}(\xi)}{\partial (s,\xi)}
=
-e^{-\int_s^t \nabla\cdot u_r\left(\Phi_{t,r}(x)\right) \dd r}u_s(\xi)\cdot n(\xi)
\end{equation}
by Proposition~\ref{p:dxPhi2} and the additional factor of $-u_s(\xi)\cdot n(\xi)$ comes
from replacing the \sX direction perpendicular to $\Gamma_\mathrm{in}$ with $s$ ($\xi$ lives in the $d-1$ dimensional
manifold $\Gamma_\mathrm{in}$) so
\begin{multline}
\int_0^t \left<\nabla\cdot F,\widetilde{A}^{t,s}[c]I_\mathrm{bdry}(s)\right>\dd s =\\
  -\int_\sX \int_\sY \left(\mathcal{S}^t \nabla\cdot F\right)(x,y) 
  e^{\int_s^t \nabla\cdot u_r\left(\Phi_{t,r}(x)\right) \dd r}\left(u_{s(x)}(\xi(x))\cdot n(\xi(x))\right)^{-1}
  I_\mathrm{bdry}(s(x),\xi(x),\dd y)\dd x\\
  = \left<\mathcal{S}^t \nabla\cdot F,\nu_t\right>
\end{multline}
where
\begin{equation}
\nu_t(x,\dd y)=
  e^{\int_s^t \nabla\cdot u_r\left(\Phi_{t,r}(x)\right) \dd r}\left(u_{s(x)}(\xi(x))\cdot n(\xi(x))\right)^{-1}
  I_\mathrm{bdry}(s(x),\xi(x),\dd y).
\end{equation} 
Writing
\begin{equation}
\mathcal{S}^t \nabla\cdot F =
\nabla\cdot\left(\mathcal{S}^t F\right) - \left(\nabla\cdot\mathcal{S}^t\right) F
\end{equation}
and checking that $\nu_t \in \Mbv$ concludes the proof.
\end{proof}

To simplify the remainder of this section it will be assumed that $\sX = [0,L)\times\Gamma_\mathrm{in}$
for some $L>0$, that is, that \sX is a
something rather like a cylinder.  The results are expected to generalise, but this assumption avoids
introducing technical conditions on \sX.  In particular
$x\in\sX$ can be written as $(x_1,x_2,\cdots, x_d)$ for $x_1\in[0,L)$ and
$(x_2,\cdots,x_d)\in \Gamma_\mathrm{in}$.

\begin{prop}\label{p:bdrycon}
Assume H1 and I3 hold, $c\in L^\infty\left([0,\infty),\Moi\right)$ solves \eqref{eq:smol} with additionally
$c\in W^{1,\infty}\left([0,\infty)\times\sX,\mathcal{M}(\sY)_\mathrm{TV}\right)$.  Let $n(x)$ be the outward 
normal on $\Gamma_\mathrm{in}$, then
\begin{equation*}
u_{1,t}(x)c(t,x,\dd y) = -u_t(x)\cdot n(x) c(t,x,\dd y) = I_\mathrm{bdry}(t,x,\dd y)\quad \forall t\in\RRP,\ x\in\Gamma_\mathrm{in}.
\end{equation*}
\end{prop}
\begin{proof}
Consider \eqref{eq:smol} with $f(x_1,x_2,\dotsc,x_d)= \frac{\epsilon - x_1}{\epsilon}\II\left\{x_1\leq \epsilon\right\}$ as
$\epsilon\rightarrow 0$.
\end{proof}

The existence of a one or more inverses to the divergence operator is necessary to avoid making
statements about an empty set of functions in the remainder of this section.
\begin{prop}\label{p:divinv}
Let $f\in D$, then there exists $g\in D^d$ such that $\nabla\cdot g \equiv f$ and
$g\cdot n = 0$ on $\Gamma_\mathrm{side}$, where $n$ is the outward normal.
\end{prop}
\begin{proof}
Take $g(x_1,x_2,\dotsc,x_d,y)=(g_1,g_2,\dotsc,g_d)$ where
\begin{equation}
g_1(x_1,x_2,\dotsc,x_d,y) = -\int_{x_1}^L f(\xi,x_2,\dotsc,x_d,y)\dd \xi
\end{equation}
and $g_i\equiv 0$ for $i>1$.
This construction has a natural generalisation in terms of path integrals.  It is not important
exactly which end point on $\Gamma_\mathrm{out}$ is chosen because $f=0$ all along this boundary.
\end{proof}
This representation is not in general unique.  Consider for example the case where $f=0$ on
$\Gamma_\mathrm{side}$ and take integrals along lines perpendicular to the direction used
in the above proof.

If $c$ solves \eqref{eq:smol}, $f\in\Wopb$ and $g\colon\sX\rightarrow C_\mathrm{b}(\sY)^d$ is
differentiable with $\nabla \cdot g \equiv f$ and $g\cdot n = 0$ on $\Gamma_\mathrm{side}$ for
normal vectors $n$,
then applying the divergence theorem to \eqref{eq:smol}
(at this stage in a purely formal calculation) suggests
\begin{multline}\label{eq:divc}
\deriv{t}\int_{\sX\times\sY} g\cdot \nabla c \,\dd x \dd y=
\int_{\sX\times\sY} u_t^\top(\nabla g)\nabla c \,\dd x \dd y
- \int_{\sX\times\sY} g^\top (\nabla u_t) \nabla c \,\dd x \dd y \\
- \int_{\sX\times\sY} g \cdot \nabla (\nabla \cdot u_t) c \,\dd x \dd y
+ \int_{\sX\times\sY} g \cdot \nabla I_\mathrm{int} \,\dd x \dd y\\
+ \int_{\Gamma_\mathrm{in}\times\sY} g \cdot n \pderiv{t}c \,\dd x \dd y
- \int_{\Gamma_\mathrm{in}\times\sY} u_t^\top (\nabla g) n c \,\dd x \dd y
- \int_{\Gamma_\mathrm{in}\times\sY} g \cdot n I_\mathrm{int} \dd x \dd y \\
- \int_{\Gamma_\mathrm{in}\times\sY} \nabla\cdot g I_\mathrm{bdry} \,\dd x \dd y
+ \int_{\Gamma_\mathrm{in}\times\sY} g^\top (\nabla u_t) n c \,\dd x \dd y\\
+\frac12 \int_{\sX\times\sY} \int_{\Gamma_\mathrm{in}\times\sY}K(y,y_2)g(x,y+y_2)\cdot n(x) \left( h(x,x_2)
  c_t(x,\dd y)\right)c_t(x_2,\dd y_2) \dd x_2\dd x\\
-\frac12 \int_{\sX\times\sY} \int_{\sX\times\sY}K(y,y_2)g(x,y+y_2)\cdot \nabla\left( h(x,x_2)
  c_t(x,\dd y)\right)c_t(x_2,\dd y_2) \dd x_2\dd x\\
-\frac12 \int_{\sX\times\sY} \int_{\Gamma_\mathrm{in}\times\sY}K(y,y_2) g(x,y)\cdot n(x)
  \left(\left(h(x,x_2)+h(x_2,x)\right) c_t(x,\dd y)\right)c_t(x_2,\dd y_2) \dd x_2\dd x\\
+\frac12 \int_{\sX\times\sY} \int_{\sX\times\sY}K(y,y_2)g(x,y)\cdot
  \nabla\left(\left(h(x,x_2)+h(x_2,x)\right)c_t(x,\dd y)\right)c_t(x_2,\dd y_2) \dd x_2\dd x.
\end{multline}

\begin{defn}
Define a norm on $\Mi^d$, which by a slight abuse of notation will also be referred to as the $\Lpb^\star$-norm by setting
$\norm{\mu}_{\Lpb^\star}=\sup_{f\in\Lpb^d\colon \normoi{f}=1} \abs{\left<f,\mu\right>}$,
for $\mu \in \Mi^d$.
\end{defn}

Conditions are now provided to make \eqref{eq:divc} rigorous, first by restricting the test functions to
the interior of the domain, so that the boundary terms can be ignored and then proceeding to more general
test functions:
\begin{prop}\label{p:derivgen}
Assume H2 and I3 hold and that $c\in L^\infty\left([0,\infty),\Moi\right)$ solves
\eqref{eq:smol}.  Suppose further that $c_t\in\Mbv$ for each $t$, so that there
exist vector measures $\nu_t$ of finite total variation such that
$\left<\nabla\cdot f, c_t\right> =-\left<f,\nu_{t}\right>$ for all
$f\in \Wopb^d$.  This means that $\nu \in L^\infty\left([0,\infty),(\Mi^d,\norm{\cdot}_{B^\star})\right)$
and in particular for all
$g\in C^1_\mathrm{K}\left(\sX^\circ,\mathcal{B}_\mathrm{b}(\sY)^d\right)$, the space of once continuously
differentiable functions with compact support strictly contained in the interior $\sX^\circ$ of \sX that also
satisfy $\nabla \cdot g \in D$
(for example $g\in C^2_\mathrm{K}\left(\sX^\circ,\mathcal{B}_\mathrm{b}(\sY)^d\right)$)
one has
\begin{multline*}
\deriv{t}\left<g,\nu_{t}\right>=
\left<(\nabla g)^\top u_t ,\nu_{t}\right>-\left<(\nabla u_t)^\top g,\nu_{t}\right>-
\left<g\cdot \nabla\left(\nabla\cdot u\right),c_t\right>
+\left<g,\nabla I_{\mathrm{int},t}\right> +\left<\widehat{{H}_t[c]}g,\nu_{t}\right>\\
+\frac12 \int_{\sX\times\sY} \int_{\sX\times\sY}g(x,y+y_2) K(y,y_2)\widehat{c_t}(x,\dd y)
  \cdot\left(\nabla_x h(x,\xi)(\dd x)\right)c_t(\xi,\dd y_2)\dd \xi\\
- \frac12 \int_{\sX\times\sY}\int_{\sX\times\sY}g(x,y)K(y,y_2)  \widehat{c_t}(x,\dd y)
    \cdot\left(\nabla_x \left[h(x,\xi)+h(\xi,x)\right](\dd x)\right)c_t(\xi,\dd y_2)\dd \xi.
\end{multline*}
Here $\widehat{{H}_t[c]}$ is a bounded linear operator mapping $\Lpb \rightarrow \Lpb$ acting
componentwise on $g$ with
$\norm{\widehat{H_t[c]}}_{\Lpb\rightarrow\Lpb} = \norm{H_t[c]}_{\Lpb\rightarrow\Lpb}$
and $\widehat{H_t[c]}g(x,y) = H_t[c]g(x,y)$ (recall $H_t$ is specified in Definition~\ref{d:Htstar})
for all $y$ and all $x$ except possibly $x$ at which $h$ and
$c$ both have discontinuities, which is a set of ($\RR^d$-Lebesgue) measure 0.   Similarly
$\normmoi{\widehat{c}}=\normmoi{{c}}$ with possible differences between $c$ and $\widehat{c}$
on the same set of measure 0.
Because $h$ is only assumed to be of bounded variation, it only has a weak derivative; in the case of
the weak derivative with respect to the first argument this is written $\nabla_x h(x,\xi)(\dd x)$. 
\end{prop}
\begin{proof}
The boundary integrals on
$\Gamma$ vanish because $g$ is zero here.  Note that the product of two functions of bounded variation
($c$ and the $h$ in the definition of $H$)
is itself of bounded variation, but the Leibniz product rule for differentiation has to be adapted slightly at
points where both  are discontinuous (yielding $\widehat{H}$ and $\widehat{c}$). The details follow
from \citep[Theorem~3.96 \& Example~3.97]{Amb00}.  In one dimension this amounts to adjustments
to give left or right continuity at the jump points.
\end{proof}
The terms in the preceding expression can be grouped as follows ($c$ is in this context known):
\begin{itemize}
\item Transport $\left<u_t^\top\nabla g,\nu_t\right>=\left<U_t g, \nu_{t}\right>$ (note $\nabla g$ is a matrix),
\item linear reactions $-\left<g \nabla u_t,\nu_{t}\right> +\left<\widehat{H_t[c]}g,\nu_{t}\right> $,
\item source terms, which are collected as a vector measure $\widehat{J}_{t}[c]$ so that,
for $g\in C_\mathrm{K}\left(\sX^\circ,\mathcal{B}_\mathrm{b}(\sY)^d\right)$
\begin{multline}\label{e:jint}
\left<g,\widehat{J}_{t}[c]\right>=
-\left<g\cdot \nabla\left(\nabla\cdot u\right),c_t\right>
+\left<g,\nabla I_{\mathrm{int},t}\right> \\
+\frac12 \int_{\sX\times\sY} \int_{\sX\times\sY}g(x,y+y_2)K(y,y_2)
     \widehat{c_t}(x,\dd y)\left(\nabla_x h(x,\xi)(\dd x)\right)c_t(\xi,\dd y_2) \dd \xi\\
- \frac12 \int_{\sX\times\sY}\int_{\sX\times\sY}g(x,y)K(y,y_2)
     \widehat{c_t}(x,\dd y) \left(\nabla_x \left[h(x,\xi)+h(\xi,x)\right](\dd x)\right)c_t(\xi,\dd y_2) \dd \xi.
\end{multline}
\end{itemize}
This characterisation is however limited to functions with compact support in the interior of \sX.  It
can only give information about how a solution changes within \sX, it says nothing about what might
happen on $\Gamma_\mathrm{in}$.  Including the boundary terms in the integration by parts/Gauss
Theorem used for Proposition~\ref{p:derivgen} yields the following additional terms.  That these are
the correct additional terms is part of the assertion of Proposition~\ref{p:uniqderiv}.

\begin{defn}\label{d:derivsrc}
Let $c\in L^\infty\left([0,\infty),\Moi\right)$ and define a vector measure $J_t[c]$
on $\sX\times\sY$ by
\begin{multline*}
\left<g,J_t[c]\right> = \left<g,\widehat{J}_t[c]\right> +\\
\int_{\Gamma_\mathrm{in}\times\sY}g(x,y)\cdot n(x)  
  \left(\pderiv{t} \frac{I_\mathrm{bdry}(t,x,\dd y)}{u_t(x)\cdot n(x)} +I_\mathrm{int}(t,x,\dd y)
          -\nabla_\mathrm{\Gamma}\cdot\left(\frac{u_t(x)I_\mathrm{bdry}(t,x,\dd y)}{u_t(x)\cdot n(x)}\right)
          \right)\dd x\\
-\frac12 \int_{\sX\times\sY} \int_{\Gamma_\mathrm{in}\times\sY}K(y,y_2) g(x,y)\cdot n(x)
  \left(\left(h(x,x_2)+h(x_2,x)\right) \right)\frac{I_\mathrm{bdry}(t,x,\dd y)}{u_t(x)\cdot n(x)}c_t(x_2,\dd y_2) \dd x_2\dd x\\
          +\int_{\Gamma_\mathrm{in}\times\sY}g(x,y)\cdot \nabla_\mathrm{\Gamma}I_\mathrm{bdry}(t,x,\dd y)\dd x
          - \int_{\Gamma_\mathrm{in}\times\sY}\frac{I_\mathrm{bdry}(t,x,\dd y)}{u_t(x)\cdot n(x)}
          g(x,y)\cdot \left(\nabla_\mathrm{\Gamma}n(x)\right)^\top u_t(x) \dd x
\end{multline*}
for $g \in \Lpb^d$ and
where $\nabla_\mathrm{\Gamma}$ is the derivative restricted to directions perpendicular to $n(x)$. 
Under the assumptions on \sX set out above $\nabla_\Gamma = (0,\pderiv{x_2},\dotsc,\pderiv{x_d})$.
\end{defn}

\begin{defn}
For $c\in L^\infty\left([0,\infty),\Moi\right)$ define time dependent linear operators $\widetilde{G}_t[c]$
on $\Mi^d$ by $\left<g,\widetilde{G}_{t}[c]\nu\right> = - \left<g\cdot\nabla u_t,\nu\right> + \left<\widehat{H_t[c]}g,\nu\right>$ for all $g \in \Lpb^d$.
\end{defn}

One can now compactly rewrite the equation from Proposition~\ref{p:derivgen} as (compare \eqref{eq:Atstar})
\begin{equation}\label{e:deriv}
\deriv{t}\left<g,\nu_t\right>=\left<g,\widetilde{U}_t \nu_t\right>+\left<g,\widetilde{G}_t[c]\nu_t\right>+\left<g,J_t[c]\right>
\end{equation}
for all $g\in\Lpb^d$ such that $\nabla\cdot g \in D$.
The additional terms introduced in Definition~\ref{d:derivsrc} are not seen by the smaller class of test
functions used in Proposition~\ref{p:derivgen}.

\begin{prop}
Assume H1 holds and that $c\in L^\infty\left([0,\infty),\Moi\right)$,
then there is a strongly continuous, bounded 
propagator $\widetilde{V}^{t,s}[c]$ on $\Mi^d$ with
\begin{equation*}
\norm{\widetilde{V}^{t,s}[c]}_{\Mi^d\rightarrow\Mi^d} \leq e^{\left(\frac32 K_\infty \Cho \opnorm{c}_{\Lpb^\star}+\opnorm{\nabla u}\right)\min(t-s,t_0)}
\end{equation*}
and for $f\in\Wopb$, $\mu \in\Mi^d$
\begin{equation*}
\deriv{t}\left<f,\widetilde{V}^{t,s}[c]\mu\right>=\left<f,\left(\widetilde{U}_t + \widetilde{G}_t[c]\right)\widetilde{V}^{t,s}[c]\mu\right>.
\end{equation*}

\end{prop}
\begin{proof}
This follows the same perturbation argument as Proposition~\ref{p:ApropB} since by duality\\
$\norm{\widetilde{G}_t[c]}_{\Mi^d\rightarrow\Mi^d} \leq \frac32 K_\infty \Cho \opnorm{c}_{\Lpb^\star}+\opnorm{\nabla u}$.
\end{proof}

\begin{prop}\label{p:uniqderiv}
Let I3 hold; assume further that either $d=1$, H2 holds and $\inf_{t,x} u_t(x)>0$ or H3 holds for general $d$;
assume further that $c\in L^\infty\left([0,\infty),\Moi\right)$,
then \eqref{e:deriv} has a unique solution
\begin{equation*}
\nu_t = \widetilde{V}^{t,0}[c]\nu_0 + \int_0^t \widetilde{V}^{t,s}[c] J_s[c] \dd s
\in C\left([0,\infty),\left(\Mi^d,\norm{\cdot}_{\Lpb^\star}\right)\right)
\end{equation*}
with initial condition $\nu_0$.
This solution is in $L^\infty\left([0,\infty),\Moi^d\right)$ provided $\nu_0\in\Moi^d$ and thus (identifying the
measure with its \sX-density) also in
$L^\infty\left([0,\infty)\times\sX,\mathcal{M}(\sY)_\mathrm{TV}^d\right)$.
\end{prop}
\begin{proof}
Existence and uniqueness are immediate for this linear problem.  Continuity in the $\Lpb^\star$-norm
follows from the strong continuity
in $t$ of $\widetilde{V}^{t,s}$.

That the propagators $\widetilde{V}^{t,s}$ preserve $\Moi^d$ can be seen by analogy with Proposition~\ref{p:densprop}.
Definition~\ref{d:derivsrc} expresses the $J_t[c]$ as a sum of $\widehat{J}_t[c]$ and a term
concentrated on the inflow boundary.  Under H3 $\widehat{J}_t[c]\in\Moi^d$ and so one argues
as in Propositions~\ref{p:bdrydens2}\&\ref{p:bdrydens} to
show that $\int_0^t \widetilde{V}^{t,s}[c] J_s[c] \dd s$ has a density with respect to Lebesgue measure on \sX.

In the case when only H2 holds, then the $x$-derivatives of $h$ in \eqref{e:jint} may only exist in a
distributional sense.  However, under H2, the measure $\nabla_x h(x,\xi)(\dd x)$ can be expressed as
a sum of an absolutely continuous part with a bounded density and a finite number of atoms
$\alpha_k(t) \delta_{a_k}$ with $\alpha_k(t) \in \RR$, $a_k \in \sX\subset \RR$.  When $d=1$ each of
these atoms is like a simpler version of the boundary part of the inception measure, which in this case
reduces under the assumption I2 (see \S~\ref{s:incep}) to $I_\mathrm{bdry}(t,\dd y)\delta_0(\dd x)$.
The boundedness (uniform in $t$ and $k$) of the $\alpha_k(t)$ is immediate from the boundedness
of $K$ and $c$ and since $u$ is bounded away from 0 the analysis of
Propositions~\ref{p:bdrydens2}\&\ref{p:bdrydens} applies to show that for each $k$ and all $t$
\begin{equation}
\int_0^t \widetilde{V}^{t,s}[c]  \alpha_k(t) \delta_{a_k}\dd s \in \Moi^d
\end{equation}
with a global in time bound in the $\Moi$-norm.
\end{proof}

\begin{proof}[Proof of Theorem \ref{t:globdiff}]

For $d=1$ assume without loss of generality that $\sX=[0,L)$ for some $L>0$ and
$\Gamma_\mathrm{in}=\left\{0\right\}$.
The boundary condition $c(t,0,\dd y) =\frac{I_\mathrm{bdry}(t,0,\dd y)}{u_t(0)}$ is given by
Proposition~\ref{p:bdrycon}.
The presumed derivative $\nu$ from Proposition~\ref{p:uniqderiv} is then used to construct
\begin{equation}
\widetilde{c}(t,x,\dd y) =
\frac{I_\mathrm{bdry}(t,0,\dd y)}{u_t(0)} 
+
\int_0^x \nu(t,\xi,\dd y)\dd \xi,
\end{equation}
which is readily seen to be a strong solution to \eqref{eq:smol} and therefore to be in the same
$L^\infty\left([0,\infty),\Moi\right)$ equivalence class as $c$.  Therefore (a version of) $c$ is
in $L^\infty\left([0,\infty),W^{1,\infty}\left(\sX,\mathcal{M}(\sY)_\mathrm{TV}\right)\right)$ and since
$\deriv{t}c$ can be expressed in terms of $c$ and $\deriv{x}c$ the result follows.

This argument does not generalise easily to more then one space dimension.  However the existence
of a weak derivative was shown in Proposition~\ref{p:solnbv} and under H3 Proposition~\ref{p:uniqderiv}
shows that this weak derivative in fact has an $L^\infty$ density.  The boundary condition comes from Proposition~\ref{p:bdrycon}.
\end{proof}

\section{Discussion}
This paper proves the well posedness of an equation for measures, modelling the creation and coagulation of particles
in a flow, for example a flame, for which stochastic approximations were studied in \citep{Pat13}.  In that
work the existence of one or more non-negative solutions was proved under somewhat less general
assumptions there by constructing the solutions as limits of stochastic approximations.  The present
work extends this result by showing that there is in fact only one solution to the equation for a given
initial condition and thus that all limit points of the approximating
sequence from \citep{Pat13} are the same and those approximations converge rather than merely having
convergent sub-sequences.  The present work incidentally provides an additional, less constructive proof
of the existence of a solution to \eqref{eq:smol}.

It is proved here and in \citep{Pat13} that solutions to \eqref{eq:smol} have a density with respect to
Lebesgue measure on \sX and that this is uniformly bounded in time and in \sX.  The differentiability
of the density is established here even for delocalisations that are of bounded variation, but only in
one spatial dimension.  This result does not extend in full generality to higher spatial dimensions---it
is easy to imagine two parallel streams of particles that never mix and therefore not even continuity
over the dividing line in the flow, much less differentiability, is to be expected.  It seems therefore likely
that the discontinuous, cell based delocalisation of the coagulation interaction used for numerical
purposes in \citep{Pat12} is not well suited to more than one spatial dimension and that smoother
delocalisations should be used.  Similar methods have been used for the simulation of Boltzmann gases\citep{Ols08}.

\subsection{Acknowledgements}
The author thanks his colleagues Marita Thomas and Michiel Renger for their advice and patience.

\bibliography{./bib/rob-references}

\appendix
\section{The flow field}\label{s:flow}
\begin{defn}\label{d:Phi}
Let $s,t \in \RR$ and define the flows $\Phi_{s,t}$ by
\begin{equation*}
\pderiv{t}\Phi_{s,t}(x) = u_t\left(\Phi_{s,t}(x)\right),
 \quad \Phi_{s,s}(x) = x.
\end{equation*}
\end{defn}

$\Phi$ is a vector, so in more than one dimension it is necessary to distinguish between
the matrix $\nabla \Phi$, which is the subject of the next two propositions and the divergence, a
real number $\nabla \cdot \Phi$, which occurs in connection with the velocity field $u$.
\begin{prop}\label{p:dsPhi}
\begin{equation*}
\pderiv{s}\Phi_{s,t}(x)= -\nabla \Phi_{s,t}(x)u_s(x).
\end{equation*}
\end{prop}
\begin{proof}
\begin{multline}
\lim_{\delta \searrow 0} \frac{\Phi_{s,t}(x)-\Phi_{s-\delta,t}(x)}{\delta}
= \lim_{\delta \searrow 0}  \frac{\Phi_{s,t}(x)-\Phi_{s,t}\left(\Phi_{s-\delta,s}(x)\right)}{\delta}\\
=\lim_{\delta \searrow 0}   \frac{\nabla \Phi_{s,t}(x)\left(x-\Phi_{s-\delta,s}(x)\right)}{\delta}
= -\nabla \Phi_{s,t}(x)u_s(x).
\end{multline}
The right sided limit is dealt with similarly.
\end{proof}

\begin{prop}\label{p:dxPhi2}
\begin{equation*}
e^{-\opnorm{\nabla u} (t-s)}
\leq \norm{\nabla \Phi_{s,t}(x)}_{\RR^d\rightarrow\RR^d}
\leq e^{\opnorm{\nabla u} (t-s)}.
\end{equation*}
and
\begin{equation*}
\det \nabla \Phi_{s,t}(x) =
 e^{\int_s^t \nabla\cdot u_r\left(\Phi_{s,r}(x)\right)\dd r}, \qquad
\det \nabla \Phi_{t,s}(x) =
 e^{-\int_s^t \nabla\cdot u_r\left(\Phi_{t,r}(x)\right)\dd r}.
\end{equation*}
\end{prop}
\begin{proof}
For the first statement one has $\pderiv{t}\Phi_{s,t}(x) = u_t\left(\Phi_{s,t}(x)\right)$ so that, since $u$ and
therefore $\Phi$ are both smooth,
\begin{equation}
\pderiv{t}\nabla \Phi_{s,t}(x) = \nabla u_t\left(\Phi_{s,t}(x)\right) \nabla \Phi_{s,t}(x)
\end{equation}
and the result follows by an application of Gronwall's inequality.

The result for the determinant is known as Liouville's formula.  One checks by row operations that
$\det  \nabla u_t\left(\Phi_{s,t}(x)\right) \nabla \Phi_{s,t}(x)  =  \mathrm{Tr} (\nabla u) \det \left(\nabla \Phi_{s,t}(x)\right)$ and the result that follows by solving the resulting ODEs.
\end{proof}

\end{document}